\documentclass[a4paper, 11pt,reqno]{amsart}

\usepackage{amsmath}
\usepackage{amssymb}
\usepackage{hyperref} 
\usepackage{amsfonts}
\usepackage{mathtools}
\usepackage{theoremref}
\usepackage{amsthm}
\usepackage[dvipsnames]{xcolor}
\usepackage[utf8]{inputenc}
\usepackage{array}
\usepackage{framed, enumitem}
\usepackage[english]{babel}
\usepackage{comment}
\usepackage{multicol}

\calclayout

\setlist{topsep=0mm,partopsep=0mm,itemsep=1mm}

\theoremstyle{plain}
\newtheorem{lemma}{Lemma}[section]
\newtheorem{thm}[lemma]{Theorem}
\newtheorem{cor}[lemma]{Corollary}
\newtheorem{prop}[lemma]{Proposition}
\newtheorem*{mainthm}{Main Theorem}

\theoremstyle{definition}
\newtheorem*{example}{Example}

\theoremstyle{remark}
\newtheorem{cclaim}{Claim}[lemma]

\newtheorem{rem}[lemma]{Remark}

\DeclareMathOperator{\rank}{\textup{\textsf{rank}}}
\newcommand{\rk}[1]{\rank(#1)}

\let\ker\undefined
\DeclareMathOperator{\ker}{\textup{\textsf{ker}}}
\newcommand{\ke}[1]{\ker(#1)}

\DeclareMathOperator{\image}{\textup{\textsf{im}}}
\newcommand{\im}[1]{\image(#1)}

\DeclareMathOperator{\dom}{\textup{\textsf{dom}}}
\newcommand{\dm}[1]{\dom(#1)}

\DeclareMathOperator{\depth}{\textup{\textsf{depth}}}
\newcommand{\dep}[1]{\depth(#1)}

\newcommand{\PT}{\mathcal{PT}}
\newcommand{\T}{\mathcal{T}}
\newcommand{\I}{\mathcal{I}}

\newcommand{\D}{\mathcal{D}}
\newcommand{\J}{\mathcal{J}}
\renewcommand{\L}{\mathcal{L}}
\newcommand{\R}{\mathcal{R}}
\renewcommand{\H}{\mathcal{H}}

\newcommand{\C}{\mathcal{C}}

\newenvironment{thmenumerate}{\begin{enumerate}[label=\textup{(\roman*)},leftmargin=10mm]}{\end{enumerate}}
\newenvironment{nitemize}{\begin{itemize}[label=\textbullet, leftmargin=5mm]}{\end{itemize}}

\begin{document}

\title[Relational depth]{Relational depth of transformation semigroups and their ideals}
\author{N.\ Ru\v{s}kuc}
\address{School of Mathematics and Statistics, University of St Andrews, St Andrews, Scotland, UK}
\email{nik.ruskuc@st-andrews.ac.uk}

\author{Z.\ Yayi}
\address{School of Mathematics and Statistics, University of St Andrews, St Andrews, Scotland, UK}
\email{yz201@st-andrews.ac.uk}

%t

\begin{abstract}
We introduce the concept of relational depth of a finite semigroup $S$ whose $\J$-classes form a chain. It captures how far down in the ideal structure one is obliged to go in order to define the semigroup by generators and defining relations. We determine the exact value for the relational depth of an arbitrary ideal in the full transformation monoid, symmetric inverse monoid and in the partial transformation monoid.
\end{abstract}

\thanks{}
\maketitle

\section{Introduction}

Generating sets and defining relations (presentations) for classical transformation semigroups -- such as the partial transformation semigroup $\PT_n$, the full transformation semigroup $\T_n$, and the symmetric inverse semigroup $\I_n$ -- have over time received a lot of attention. 
For example, it is known that each of them is generated by the symmetric group (i.e. its group of units) together with one or two transformations of rank $n-1$; see \cite[Section 3.1]{CFTS}.
Presentations for these semigroups were first established in \cite{Popova61,Aizenstat58,Sutov60} respectively.
Another presentation for $\I_n$ was given in \cite{Meakin93}, and there is a nice exposition in \cite[Chapter 9]{CFTS}.
All the known presentations for these semigroups have the property that every defining relation represents a transformation of rank $n$, $n-1$ or $n-2$.

Ideals of these semigroups have likewise been of interest. It is known that the ideal $I_m$ consisting of all transformations of rank $\leq m$ is generated by transformations of rank precisely $m$; see \cite[Lemma 4]{HM:IR}. Nice presentations for these ideals are much harder to find. There is some work on the largest proper ideal $I_{n-1}$ \cite{East06, East15, East10a, East13, East10b, East14}.
These presentations are already quite large, but have the intriguing property that all the defining relations represent elements of ranks $n-1$ and $n-2$ only.

Very recent relevant work includes:
\begin{itemize}
\item
\cite{MW:SP} where Mitchell and Whyte investigate, and in certain cases determine, shortest (in terms of the number of defining relations) presentations for $\T_n$, $\I_n$ and $\PT_n$.
\item
\cite{CDEGZ} where the authors discuss a theoretical underpinning for some of the East's singular presentations for \cite{East06, East15, East10a, East13, East10b, East14}.
\end{itemize}

The purpose of this paper is to investigate systematically how `deep' inside the ideal structure of $I_m$ one must go in order to define the semigroup $I_m$.

We now make this idea more formal. Throughout this paper, we will consider semigroups $S$ whose $\J$-classes form a chain $J_\epsilon < J_1< \dots < J_k$. Here, typically, $\epsilon\in \{0,1\}$. The reason for this slightly strange indexing is that in our specific examples there there are  natural indexings of $\J$-classes (by rank), but they, unfortunately, do not have the same starting point. The concept of $\J$-classes, and all the other necessary prerequisites will be introduced in Section \ref{sec:prelim}.

Consider such a semigroup $S$, and suppose that $\mathcal{P}=\langle A \ |\ \R\rangle$ is a presentation for $S$
Successively define the \emph{depth} of various objects as follows:

\begin{nitemize}
\item
a $\J$-class $J$: $\dep{J_i}=k-i+1$;
\item
an element $s\in S$: $\dep s:=\dep {J_i}$, where $s\in J_i$;
\item
a word $w\in A^+$:
$\dep{w}$ is equal to the  depth of the element represented by~$w$; 
\item
a relation $u=v$ in $R$: $\dep{u=v}:=\dep{u}=\dep{v}$;
\item
the presentation $\mathcal{P}$:
$\dep{\mathcal{P}}$ is the maximal depth of a defining relation in $\R$;
\item
the semigroup $S$: 
 \[
 \dep{S}:=\min\{\dep{\mathcal{P}}\colon \mathcal{P} \text{ is a presentation for  } S\}.
 \]
\end{nitemize}

We also refer to this final quantity as the \emph{relational depth} of $S$.
The purpose of this article is to determine the
relational depth of transformation semigroups $\PT_n, \T_n, \I_n$, and their ideals.

Let $S\in\{\PT_n,\T_n,\I_n\}$.
It is known that the $\J$-classes of $S$
are of the form 
\[
J_r=\{\alpha\in S\colon\rk{\alpha}=r\}\quad\text{for } \epsilon\leq r\leq n.
\]
Here $\epsilon=0$ when $S=\PT_n$ or $S=\I_n$, and 
$\epsilon=1$ when $S=\T_n$.
The ideals of $S$ are 
\[
I_m=\{\alpha\in S\colon \rk{\alpha}\leq m\}\quad\text{for } \epsilon\leq m\leq n.
\] 
The $\J$-classes of $I_m$ are precisely the $\J$-classes of $S$ contained in $I_m$, i.e.
$J_\epsilon<J_1<\dots <J_m$.
For $\epsilon \leq i\leq m$ the depth of $J_i$ in $I_m$ is $m-i+1$.

Given a semigroup $S$, the depth of any presentation $\mathcal{P}$ for $S$ gives an upper bound for the depth of $S$. For example, in A\u{\i}zen\v{s}tat's presentation \cite{Aizenstat58}  for $\T_n$, the generating set and the set of defining relations contain maps of ranks in $[n-2,n]$. Hence A\u{\i}zen\v{s}tat's presentation has relational depth $3$. In East's  presentation \cite{East10a} for the singular part of $\T_n$, which is the ideal $I_{n-1}$ in our notation, the ranks involved are $n-1$ and $n-2$. The maps in the presentation  are in $J_{n-1}\cup J_{n-2}$, therefore East's presentation has relational depth $2$. Meakin's presentation \cite{Meakin93} for $\I_n$ has depth $3$. The presentation for $\PT_n$ given by Ganyushkin and Mazorchuk \cite{CFTS} has depth $3$. East's \cite{East06} presentation for the ideal $I_{n-1}$ of $\I_n$ has depth $2$. The above give an upper bound for the relational depths of these semigroups, and we will see that they are indeed the depths later on.

We can now state our main result:

\begin{mainthm}
\label{thm: main}
    Let $n\geq 3$, and let $S\in \{\I_n,\T_n,\PT_n\}$. Define $\epsilon =0$ for $S\in\{\PT_n,\I_n\}$ and $\epsilon=1$ when $S=\T_n$. For $m\in \{\epsilon,\dots,n\}$, let $I_m$ be the ideal of all mappings in $S$ of rank at most $m$.
    Then
    \begin{align*}
        \dep{I_m}=\begin{cases}
        m-{\rm max}(\epsilon,2m-n)+1\ &{\rm if}\ m<n\\
        3\ &{\rm if}\ m=n.
        \end{cases}
    \end{align*}
\end{mainthm}
\begin{proof}
    This will follow directly from Proposition \ref{prop: depth of semigroups same as cayley presentation} and Theorem \ref{thm: almost main theorem}.
\end{proof}

The paper is organised as follows. Section \ref{sec:prelim} contains general preliminaries not covered in this Introduction. Then in Section \ref{sec: Cayley table pre} we show that to find the relational depth of a semigroup $S$ it is sufficient to consider certain presentations arising from the Cayley table of $S$.
This leads to a reformulation of our Main Theorem as Theorem  \ref{thm: almost main theorem}.
The rest of the paper is devoted to proving this latter result. First in Section \ref{sec:low bound} we show that the stated numbers are lower bounds.
Proving that they are upper bounds  is harder. In Section \ref{sec:ubp} we establish some common methodology. And then in the next three section we treat each of the three `container semigroups' $\I_n$, $\T_n$, $\PT_n$ separately. The paper concludes with some closing remarks in Section \ref{sec:conc}.

\section{Preliminaries}
\label{sec:prelim}

In this section we introduce some definitions and notation that will be used throughout. 

For integers $m,n$, we will use the interval notation
$[m,n]$ for the set $\{ m,m+1,\dots,n\}$, and 
$[n]:=[1,n]$. 
We will consider the following three classical transformation semigroups on $[n]$:
\begin{nitemize}
\item
the \emph{partial transformation semigroup} $\PT_n$ consisting of all partial mappings; 
\item
the \emph{symmetric inverse monoid} $\I_n$ consisting of of all partial bijections;
\item
 the \emph{full transformation semigroup} $\T_n$ consisting of  all (full) mappings. 
\end{nitemize}
The operation on each of them is composition of functions.
Clearly, all three semigroups are monoids, and $\T_n$ and $\I_n$ are submonoids of $\PT_n$.

For $\alpha\in \PT_n$, the \textit{image} of $\alpha$ is $\im\alpha=\{x\alpha\colon x\in[n]\}$. The \textit{kernel} of $\alpha$ is the equivalence relation $\ke\alpha=\{(x,y)\in \dm{\alpha}\times\dm{\alpha}\colon x\alpha=y\alpha\}$, whose equivalence classes are called kernel classes, and the $\ke\alpha$-classes partition $\dm{\alpha}$. The \emph{rank} of $\alpha$ is the size of $\im\alpha$, which is also equal to the number of $\ke\alpha$-classes. 

In this paper, we will tend to write down a partial transformation $\alpha$ by specifying its kernel classes and their images under $\alpha$, together with the complement of the domain. For example, we write $\alpha=\begin{pmatrix}
    A_1 & A_2 & \dots & A_m & A\\
    a_1 & a_2 & \dots & a_m & -
\end{pmatrix}$ to mean the partial transformation of rank $m$, with domain $A_1\cup \dots\cup A_m$, 
and $x\alpha=a_i$ for 
all $x\in A_i$. 

Green's relations on a semigroup provide a standard framework for describing the ideal structure of a semigroup $S$. 
Let us write $S^1$ for the semigroup $S$ with an identity element adjoined to it, unless $S$ is already a monoid. The Green's equivalences then are:
\begin{gather*}
\L=\{ (x,y)\in S\times S\colon S^1x=S^1y\},\quad 
\R=\{ (x,y)\in S\times S\colon xS^1=yS^1\},\\
\J=\{ (x,y)\in S\times S\colon S^1xS^1=S^1yS^1\},\\
\D=\L\circ \R=\R\circ \L,\quad \H=\L\cap \R. 
\end{gather*}
In every finite semigroup $\D=\J$ holds \cite[Chapter~2]{FST}.
There is a natural partial order on the set of all $\J$-classes, given by
$J_1\leq J_2$ if and only if $S^1x S^1\subseteq S^1yS^1$, where $x\in J_1$, $y\in J_2$.

The Green's relations on a transformation semigroup $S\in\{\T_n, \I_n, \PT_n\}$ are controlled by kernels and images. Namely, for $\alpha,\beta\in S$
\begin{thmenumerate}
    \item $\alpha\L\beta\iff \im\alpha=\im\beta$;
    \item $\alpha\R\beta\iff \ke\alpha=\ke\beta$;
    \item $\alpha\D\beta\iff \rk\alpha=\rk\beta\iff \alpha\J\beta$.
\end{thmenumerate}

Thus, a  $\J$-class of $S$ consists of all elements in $S$ of a fixed rank, and we denote by $J_i$ for the $\J$-class that contains elements of rank $i$. The $\J$-classes of $S$ form a chain
under the ordering introduced above. Specifically, in 
$\PT_n$ and $I_n$ this chain is $J_0< J_1<\dots<J_n$, whereas for $\T_n$ we have
$J_1<\dots < J_n$. 

Ideals in any semigroup $S$ are downward closed unions of $\J$-classes.
Hence the ideals of $S\in \{ \PT_n,\T_n,I_n\}$ are precisely
the sets
\[
I_m:=\{ \alpha\in S\colon \rank(\alpha)\leq m\}\quad\text{for } m\in [\epsilon,n].
\]
Here $\epsilon =0$ for $S=\PT_n$ and $S=\I_n$, and $\epsilon=1$ for $S=\T_n$.

Consider a finite set $A$. A semigroup presentation is a pair $\langle A\ |\ \R\rangle$, where $A$ is an alphabet, $\R\subseteq A^+\times A^+$ is considered to be a set of defining relations, and an element $(u,v)\in\R$ is usually written as $u=v$. We say that a semigroup $S$ is defined by the presentation $\langle A\ |\ \R\rangle$ if $S \cong A^+/\R^{\sharp}$, where $\R^{\sharp}$ is the smallest congruence on $A^+$ that contains $\R$. We can obtain another word $w'$ from $w$ by replacing a subword $u$ of $w$ by $v$ given that $(u=v)\in \R$. A relation $w'=w$ is said to be a consequence of $\R$ if we can obtain $w'$ from $w$, or $w$ from $w'$ by applying the defining relations or consequences of $\R$ finitely many times.

If we know a (finite) presentation $\mathcal{P}=\langle A\ |\ \R\rangle$ for a semigroup $S$, we could obtain any another finite presentation for $S$ by applying any of the following \emph{elementary Tietze transformations} finitely many times to $\mathcal{P}$ \cite{NRthesis}:
\begin{enumerate}[label=\textup{\textsf{(T\arabic*)}},leftmargin=12mm]
    \item\label{T1} if $u=v$ is a consequence of $\langle A\ |\ \R\rangle$, add relation $u=v$ and obtain $\langle A\ |\ \R, u=v\rangle$;
    \item\label{T2} if $(u=v)\in\R$ is a consequence of $\langle A\ |\ \R\setminus\{u=v\}\rangle$, delete the relation $u=v$ from $\langle A\ |\ \R\rangle$ and obtain $\langle A\ |\ \R\setminus\{u=v\}\rangle$;
    \item\label{T3} add a new generating symbol $b$ and a new relation $b=w$ where $w\in A^+$ to $\langle A\ |\ \R\rangle$ and obtain $\langle A\cup\{b\}\ |\ \R\cup\{b=w\}\rangle$;
    \item\label{T4} if $\R$ contains a relation of the form $b=w$, where $b\in A$ and $w\in(A\setminus\{b\})^+$, then delete the generating symbol $b$, and replace any occurrence of $b$ in $\R$ by $w$.
\end{enumerate}
The analogue for groups is better documented in literature, e.g. see \cite[Chapter~6.4.3]{TGT}.

%\YZcomm{added in presentations, Tietze transformations.}

\section{Cayley Table Presentation}
\label{sec: Cayley table pre}
On the face of it, to determine the relational depth of a semigroup $S$ we need to consider all possible presentations for $S$. In this section we show that in fact it is sufficient to consider the Cayley table of $S$, viewed as a presentation, and certain restrictions of it. The \textit{Cayley table presentation} for semigroup $S$ is \begin{center}
    $\C=\C_S=\langle x_s\ (s\in S)\ |\ x_sx_t=x_{st}\ (s,t\in S)\rangle$.
\end{center}
For $U\subseteq S$ the \emph{restriction} of the Cayley table presentation restricted to $U$ is \begin{center}
    $\C_{S,U}=\langle x_s\ (s\in U)\ |\ x_sx_t=x_{st}\ (s,t,st\in U)\rangle$.
\end{center}
Notice that $\C_{S,U}$ need not define $S$. Firstly, $U$ may not be a generating set. And, secondly, even if $\langle U\rangle=S$, the relations may not be sufficient, and may only define a homomorphic preimage of $S$. 

As before, let $S$ be a semigroup whose $\J$-classes form a chain $J_\epsilon< \dots< J_k$. 
For $i\in [\epsilon, k]$, let $\C_i=\langle X_{i,k}\mid \R_{i,k}\rangle$ be the restriction of $\C_S$ to $U_i:=J_i\cup\dots\cup J_k$.

As the main result of this section we prove that to find the relational depth of $S$ it is sufficient to consider these restrictions:

\begin{prop}\label{prop: depth of semigroups same as cayley presentation}
    Let $S$ be a semigroup whose $\J$-classes form a chain $J_\epsilon<\dots  <J_k$. The relational depth of $S$ is $k-i+1$ where $i$ is the largest number for which $\C_i$ defines~$S$.
\end{prop}

To prove the proposition, we will need the following lemma.
\begin{lemma}\label{lem: complement of ideal}
    Let $S$ be a semigroup and $I\subseteq S$ an ideal. Define $U:=S\setminus I$. Suppose that $S$ has a presentation $\langle A \mid \R \rangle$ such that the following hold: 
    \begin{thmenumerate}
        \item\label{it:ci1} Every $x\in A$ represents an element of $U$;
        \item\label{it:ci2} if $(u=v)$ is a relation from $\R$, then $u$ (and hence $v$) represents an element of $U$.
    \end{thmenumerate}
    Then $\C_{S,U}$ defines $S$.
\end{lemma}

\begin{proof}
    Let $\C_{S,U}=\langle X\ |\ \mathcal{Q}\rangle$ as defined above. First, without loss of generality, we can assume that $A\subseteq X$. Furthermore, we can in fact assume that $A=X$. Indeed, if not, we could add redundant generators $X\setminus A$ to obtain another presentation which has this property, and still satisfies \ref{it:ci1}, \ref{it:ci2}. We prove that we can obtain $\langle X\mid \mathcal{Q}\rangle$ from $\langle X\mid \R\rangle$ using Tietze transformations.
    
    If $s,t\in U$ such that $st\in U$, we can deduce the relation $x_sx_t=x_{st}$ from $\R$. Adding all relations of this form to $\langle X\mid \R\rangle$ we obtain $\langle X\mid \R,\mathcal{Q}\rangle$, (transformation \ref{T1}).

    It remains to show that we can eliminate $\R$ from $\langle X\mid \R,\mathcal{Q}\rangle$. Let $(u=v)\in\R$, and suppose that $u\equiv x_1\dots x_m$, $v\equiv y_1\dots y_n$, where $x_1\dots x_m=y_1\dots y_n=x\in U$. Then $x_i,y_j\in U$ and a straightforward induction shows that the relations $x=x_1\dots x_m$ and $x=y_1\dots y_n$, and hence $u=v$ are consequences of $\mathcal{Q}$. Hence relations in  $\R$ are consequences of $\langle X\mid \mathcal{Q}\rangle$, and deleting $\R$ we obtain $\langle X\mid \mathcal{Q}\rangle$ (transformation \ref{T3}).
\end{proof}

\begin{proof}[Proof of Proposition \ref{prop: depth of semigroups same as cayley presentation}]
 Let $\dep{S}=d$. 
 Since $\dep{\C_i}=k-i+1$, we have $k-i+1\geq d$.
 On the other hand, since $S$ has a presentation with depth $d$, it follows by Lemma \ref{lem: complement of ideal} that $\C_{k-d+1}$ defines $S$. Therefore $i\geq k-d+1$, and the result follows.
\end{proof}

Now consider $S$ to be a proper ideal $I_m$ of $S\in\{\T_n,\I_n,\PT_n\}$, whose $\J$-classes are $J_{\epsilon}<\dots< J_m$. The question becomes: for which $i$ does $\C_i$ define $I_m$?

We aim to prove the following theorem.

\begin{thm}
\label{thm: almost main theorem}
Let $n\geq 3$, and let $S\in \{ \PT_n,\T_n,\I_n\}$.
Define $\epsilon =0$ for $S\in\{\PT_n,\I_n\}$ and $\epsilon=1$ when $S=\T_n$. Define $\theta=\frac{n+\epsilon}{2}$.
For $m\in [\epsilon,n)$, let $I_m$ be the  ideal of $S$ of all mappings of rank $\leq m$.
\begin{thmenumerate}
\item \label{mn1}
When $m\leq\theta$, $\C_i$ defines $I_m$ if and only if $i=\epsilon$. 
\item \label{mn2}
When $m>\theta$, $\C_i$ defines $I_m$ if and only if $i\leq 2m-n$.
\end{thmenumerate}
\end{thm}
Our Main Theorem follows from this and Proposition \ref{prop: depth of semigroups same as cayley presentation}. The rest of the paper is devoted to proving Theorem \ref{thm: almost main theorem}, which is achieved in Theorem \ref{thm: In: relational depth}, \ref{thm: Tn: relational depth}, \ref{thm: PTn: relational depth} for $\I_n,\T_n$ and $\PT_n$ respectively.

\section{Lower bound}
\label{sec:low bound}
In this section we aim to find the lower bound for the depths of proper ideals of $\I_n, \T_n$ and $\PT_n$, by proving the following proposition.

\begin{prop}\label{prop: lower bound for depth of all ideals}
    Consider $m<n$. Let $I_m$ be an ideal of $S\in\{\PT_n,\I_n,\T_n\}$. Define $\epsilon =0$ for $S\in\{\PT_n,\I_n\}$ and $\epsilon=1$ when $S=\T_n$. Then $\dep{I_m}\geq m-{\rm max}(\epsilon,2m-n)+1$.
\end{prop}

The proof follows from the next proposition, which shows that certain depths are not sufficient.

\begin{prop}\label{prop: non-pre for I_m merged for all semigroups}
    Let $m<n$. Let $I_m$ be an ideal of $S\in\{\PT_n,\I_n,\T_n\}$. Define $\epsilon =0$ for $S\in\{\PT_n,\I_n\}$ and $\epsilon=1$ when $S=\T_n$. If $m>r>{\rm max}(2m-n,\epsilon)$, then $\C_r=\langle X_{r,m}\ |\ \R_{r,m}\rangle$ does not define $I_m$.
\end{prop}

%To prove Proposition \ref{prop: non-pre for I_m merged for all semigroups}, we aim to find words that represent the same map, but whose equality cannot be derived from $\R_{r,m}$. Here is a family of more general examples.\textcolor{ForestGreen}{ I want to say that the next proposition doesn't just hold for transformation semigroups, however this sentence sounds a bit weird.}

We will derive Proposition \ref{prop: non-pre for I_m merged for all semigroups} from the following technical, but more general, result:

\begin{prop}\label{prop: general: an example of relations that cannot be deduced}
    Let $S$ be any finite semigroup whose $\J$-classes form a chain, 
    $J_\epsilon<\dots< J_m$.
    Suppose that $\alpha,\beta,\gamma\in J_m$ such that $\alpha\beta$ and $\alpha\beta\gamma=\beta\gamma\in J_r$ for some $r<m$, then $x_\alpha x_\beta x_\gamma=x_\beta x_\gamma$ is not a consequence of $\R_{r+1,m}$.
\end{prop} 
    
To prove the proposition, we first show that all words that can be obtained from such $x_\beta x_\gamma$ must be of a certain form. 

\begin{lemma}\label{lem: all: words that can be deduced from x_betax_gamma}
     Suppose that $\alpha,\beta\in J_m$ are such that $\alpha\beta\in J_r$, where $r<m$. Suppose that $w\in X_{r+1,m}^*$ is obtained from $x_\alpha x_\beta$ by applications of relations in $\R_{r+1,m}$, then \begin{equation}\label{eqn: form of word obtained from x_beta x_gamma}
    w=x_{\alpha_1}x_{\alpha_2}\dots x_{\alpha_k}x_{\beta_1}x_{\beta_2}\dots x_{\beta_l},
    \end{equation}where $\alpha_1\dots\alpha_k=\alpha$, $\beta_1\dots\beta_l=\beta$, $\alpha_i,\beta_j\in J_m$ for all $i,j$, and $\alpha_k\beta_1\in J_r$. 
\end{lemma}

\begin{proof}
    We induct on the number of applications of the relations from $\R_{r+1,m}$ to $x_\alpha x_\beta$. In the base case, no relation from $\R_{r+1,m}$ is applied, and the word $x_\alpha x_\beta$ satisfies the conditions \eqref{eqn: form of word obtained from x_beta x_gamma}.
    
     Now suppose that $w=x_{\alpha_1}\dots x_{\alpha_k}x_{\beta_1}\dots x_{\beta_l}$ is a word that satisfies \eqref{eqn: form of word obtained from x_beta x_gamma}. Let $w'$ be the word obtained from $w$ by one application of a relation from $\R_{r+1,m}$. Since all relations in $\R_{r+1,m}$ equate a letter and a word of length two, applying a relation from $\R_{r+1,m}$ necessarily changes the length, and $w'$ is either one letter longer or shorter than $w$. 

    If $w'$ is longer, a length one subword $u$ of $w$ has been replaced. If $u=x_{\alpha_i}$ where $i\in \{1,\dots,k\}$, the relation applied to $w$ has the form $x_{\alpha_i}=x_{\alpha'} x_{\alpha''}$, where $\alpha_i=\alpha'\alpha''$. Since 
    $J_m=J_{\alpha_i}\leq J_{\alpha'}$ and $J_{\alpha_i}\leq J_{\alpha''}$, $\alpha',\alpha''\in J_m$. If $i=k$, it remains to show that $\alpha''\beta_1\in J_r$. Since $\alpha''\beta_1\L \alpha'\alpha''\beta_1$ and $S$ is stable, $\alpha''\beta_1\J \alpha'\alpha''\beta_1$, and $\alpha',\alpha''\in J_m$.
    Then\begin{center}
        $w'=x_{\alpha_1}\dots x_{\alpha_{i-1}}x_{\alpha'}x_{\alpha''}x_{\alpha_{i+1}}\dots x_{\alpha_k}x_{\beta_1}\dots x_{\beta_l}$
    \end{center} satisfies \eqref{eqn: form of word obtained from x_beta x_gamma}. Similarly, if $u=x_{\beta_j}$ for $j\in \{1,\dots,l\}$, $w'$ can be shown to have the expected form. 
    
    If $w'$ is shorter than $w$, a length two subword $v$ of $w$ has been replaced by a letter to obtain $w'$. Notice that since $\alpha_k\beta_1\in J_m$ by the induction hypothesis, relations of the form $x_{\alpha_k}x_{\beta_1}=v$ are not in $\R_{r+1,m}$ Therefore the word $v$ must be one of $x_{\alpha_i}x_{\alpha_{i+1}} (i\in\{1,\dots,k-1\})$ or $x_{\beta_j}x_{\beta_{j+1}} (j\in\{1,\dots,l-1\})$.
    Without loss of generality, suppose that $v=x_{\alpha_i}x_{\alpha_{i+1}}$. The only relation applicable would be $x_{\alpha_i}x_{\alpha_{i+1}}=x_{\alpha'}$, where $\alpha'=\alpha_i\alpha_{i+1}$. Since $\alpha=\alpha_1\dots\alpha_{i-1}\alpha'\alpha_{i+2}\dots\alpha_k\in J_m$, $J_m\leq J_{\alpha'}$, and $\alpha'\in J_m$. If $i=k-1$, since $\alpha_{k-1}\alpha_k\beta_1\L \alpha_k\beta_1$ and $S$ is stable, $\alpha_{k-1}\alpha_k\beta_1\J\alpha_k\beta_1$. Therefore $\alpha'\beta_1\in J_r$.
    Then\begin{center}
        $w'=x_{\alpha_1}\dots x_{\alpha_{i-1}}x_{\alpha'}x_{\alpha_{i+2}}\dots x_{\alpha_k}x_{\beta_1}\dots x_{\beta_l}$
    \end{center}is of the form stated in \eqref{eqn: form of word obtained from x_beta x_gamma}. The case when $v=x_{\beta_j}x_{\beta_{j+1}}$ is analogous, and the induction is complete.
    \end{proof}

\begin{proof}[Proof of Proposition \ref{prop: general: an example of relations that cannot be deduced}] This follows from Lemma \ref{lem: all: words that can be deduced from x_betax_gamma}, as $x_\alpha x_\beta x_\gamma$ does not satisfy the conditions in \eqref{eqn: form of word obtained from x_beta x_gamma}.
\end{proof}

\begin{proof}[Proof of Proposition \ref{prop: non-pre for I_m merged for all semigroups}] 
It suffices to find mappings in $S\in\{\T_n, \I_n, \PT_n\}$ that satisfy the condition in Proposition \ref{prop: general: an example of relations that cannot be deduced}.
    First, consider $S=\T_n$. Let $\alpha,\beta\in J_m$ where
    \begin{align*}
        \alpha &= \begin{small}
            \begin{pmatrix}
               1 & 2 & \dots & m-1 & [m,n]\\
               1 & 2 & \dots & m-1 & m
               \end{pmatrix}
        \end{small},\\
        x\beta &= \begin{small}\begin{cases}
            x & (x\in [1,r-2]),\\
            r-1 & (x\in [r-1,m]),\\
            x-(m+1)+r & (x\in [m+1,2m-r]),\\
            m & (x\in [2m-r+1,n]).
        \end{cases}
        \end{small}
    \end{align*}Since $m>r>{\rm max}(2m-n,1)$, the $\ke\beta$-classes $[m+1,2m-r]$ and $[2m-r+1,n]$ are non-empty, and indeed $\alpha,\beta\in J_m$. With some calculation, it can be shown that
        \begin{align*}
        \alpha\beta=\beta^2= \alpha\beta^2&=\begin{small}
            \begin{pmatrix}
               1 & 2 & \dots & r-2 & [r-1,n]\\
               1 & 2 & \dots & r-2 & r-1
               \end{pmatrix}\in J_{r-1}.
        \end{small}
    \end{align*} Hence, by Proposition \ref{prop: general: an example of relations that cannot be deduced}, the relation $x_\alpha x_\beta^2=x_\beta^2$ cannot be deduced from $\R_{r,m}$, and $\C_r$ does not define $I_m$.\\
    For $S= \I_n$ or $\PT_n$, consider $r>{\rm max}(2m-n,0)$. Consider partial bijections \begin{align*}
        \alpha=\begin{small}
        \begin{pmatrix}
            a_1 & a_2 & \dots & a_m\\
            a_1 & a_2 & \dots & a_m
        \end{pmatrix}
    \end{small}, {\rm \ and\ } \beta=\begin{small}
        \begin{pmatrix}
            a_1 & \dots & a_{r-1} & b_r & \dots & b_m\\
            a_1 & \dots & a_{r-1} & b_r' &\dots & b_m'
        \end{pmatrix}
    \end{small}. \end{align*} The fact that $r>2m-n$ guarantees that we can choose $\im\beta$ such that \begin{align*}
        \im\beta\cap\ke\beta=\{a_1,\dots,a_{r-1}\} {\rm \ and\ }\alpha\beta=\alpha\beta^2=\beta^2=\begin{small}
        \begin{pmatrix}
            a_1 & \dots & a_{r-1}\\
            a_1 & \dots & a_{r-1}
        \end{pmatrix}
    \end{small}\in J_{r-1}.
    \end{align*} 
%    $Similarly, the rest follows from
%     Proposition \ref{prop: general: an example of relations that cannot be deduced}.
    From here we proceed as for $\T_n$.
\end{proof}

\begin{proof}[Proof of Proposition \ref{prop: lower bound for depth of all ideals}]
Immediate consequence of Proposition \ref{prop: non-pre for I_m merged for all semigroups}.
\end{proof}

Furthermore, we now know the exact depths of small ideals. 
\begin{cor}\label{cor: All: exact value for depth of small ideals}
    Let $I_m$ be an ideal of $S\in\{\PT_n,\I_n,T_n\}$. Define $\epsilon =0$ for $S\in\{\PT_n,\I_n\}$ and $\epsilon=1$ when $S=\T_n$. If $m\leq \frac{n+\epsilon}{2}$, then $\dep{I_m}=m+1$.
\end{cor}

\begin{proof}
    Follows from Proposition \ref{prop: lower bound for depth of all ideals}.
\end{proof}

\begin{rem}
Similar issues are treated by Mitchell and Whyte in \cite{MW:SP} in their analysis of minimal presentations in terms of the number of defining relations. For example, in Lemmas 3.6, 3.7 they show that to define the symmetric inverse monoid $\I_n$ one has to use at least two relations of depth (in our terminology) $n-1$ and at least one of depth $n-2$. Similar assertions are proved for $\T_n$ and $\PT_n$ in the subsequent sections of that paper.
\end{rem}

\section{Upper bound: Preliminaries}
\label{sec:ubp}

In this section, we state and prove some technical lemmas that will subsequently be used to find the upper bounds for the relational depths of $\I_n$, $\T_n$, and $\PT_n$. For now, consider $S\in\{\I_n,\T_n,\PT_n\}$. Then $S=J_{\epsilon}\cup \dots\cup J_n$, where
$\epsilon =0 \text{ if } S\in\{\PT_n, \I_n\} \text{, and } \epsilon=1\text{ if } S=\T_n$.

The following lemma states the three conditions for deducing relations of rank $r$ in an ideal $I_m$ using relations from $\R_{r+1,m}$.

\begin{lemma}\label{lem: All: words of rank r are cons of above}
    Let $I_m$ be an ideal of $S$, then $I_m= J_{\epsilon}\cup \dots\cup J_m$. Let $r\leq 2m-n-1$. Suppose that the following statements hold. \begin{enumerate}[label=\textup{\textsf{(P\arabic*)}},leftmargin=12mm]
        \item\label{P1} For $\alpha,\beta,\gamma,\delta\in J_{r+1}$, if $\alpha\beta=\gamma\delta\in J_r$, then we can deduce the relation $x_\alpha x_\beta=x_\gamma x_\delta$ using relations in $\R_{r+1,m}$.
        \item\label{P2} For $\alpha,\beta,\gamma\in J_{r+1}$, if $\alpha\beta,\beta\gamma,\alpha\beta\gamma\in J_{r}$, then we can obtain a word $x_{\alpha'} x_{\gamma'}$ from $x_\alpha x_\beta x_\gamma$, where $\alpha',\gamma'\in J_{r+1}$ using relations in $\R_{r+1,r+2}$.
        \item\label{P3} For $\gamma\delta\in J_{r}$, if $\rk\gamma,\rk\delta\geq r+1$, we can obtain a word $x_\alpha x_\beta$ from $x_\gamma x_\delta$, where $\alpha,\beta\in J_{r+1}$ using relations in $\R_{r+1,m}$.
\end{enumerate}Then the following are true: \begin{thmenumerate}
    \item For any $\alpha_1,\dots,\alpha_k\in J_{r+1}\cup\dots\cup J_{m}$ such that $\alpha_1\dots\alpha_k\in J_{r}$, there exist $\gamma,\delta\in J_{r+1}$ such that $\beta\gamma=\alpha_1\alpha_2\dots\alpha_k$, and $x_\beta x_\gamma=x_{\alpha_1}x_{\alpha_2}\dots x_{\alpha_k}$ is a consequence of $\R_{r+1,m}$.
    \item If $\epsilon\leq r\leq 2m-n-1$, and $w,v$ are words that represent the same rank-$r$ element, then we can obtain $v$ from $w$ using $\R_{r+1,m}$.
\end{thmenumerate}
\end{lemma}

\begin{proof}(i)
    For $w=x_{\alpha_1}x_{\alpha_2}\dots x_{\alpha_k}$, define $r(w)=\sum_{i=1}^{k}\rk{\alpha_i}$. Let $w=x_{\alpha_1}\dots x_{\alpha_k}$ be a word that satisfies the condition of the theorem. Consider the following claim. 
    \begin{cclaim}
    If $r(w)> 2r+2$, then there exist $w'=x_{\alpha_1'}\dots x_{\alpha_l'}$ such that \begin{enumerate}[label=(\alph*)]
        \item $\alpha_1',\dots,\alpha_l'\in J_{r+1}\cup \dots\cup J_m$
        \item $w=w'$ is a consequence of $\R_{r+1,m}$
        \item $r(w')<r(w)$
    \end{enumerate}
\end{cclaim}
Notice that $\sum_{i=1}^{k}\rk{\alpha_i}=2r+2 \iff k=2$ and $\rk{\alpha_1}=\rk{\alpha_2}=r+1$. So if $r(w)=2r+2$, $w$ consists of $2$ letters that represent element of rank $r+1$. Also notice that we always have $\sum_{i=1}^{k}\rk{\alpha_i}\geq 2r+2$. The claim suggests that if $w=x_{\alpha_1}\dots x_{\alpha_k}\in X_{r+1,m}^*$ which represents an element of rank $r$, we can obtain a word $w'$ such that $r(w')<r(w)$ using $\R_{r+1,m}$. By repeatedly applying the claim to the resulting word, we may reduce the sum of ranks until it gets to $2r+2$. Hence, to prove the first part of the lemma, it suffices to prove the claim. 

\begin{proof}[Proof of claim]\renewcommand{\qedsymbol}{}
    Consider the prefix $x_{\alpha_1}x_{\alpha_2}$ of $w$.

    \underline{Case 1}: $\alpha_1\alpha_2\in J_{r+1}\cup \dots\cup J_m$. $x_{\alpha_1\alpha_2}=x_{\alpha_1}x_{\alpha_2}$ is a relation in $\R_{r+1,m}$. Applying this to $w$, we obtain $w'=x_{\alpha_1\alpha_2}x_{\alpha_3}\dots x_{\alpha_k}$, and $r(w')<r(w)$ follows from the fact that $\rk{\alpha_1\alpha_2}\leq {\rm min}\{\rk{\alpha_1},\rk{\alpha_2}\}$.
    
    \underline{Case 2}: $\alpha_1\alpha_2\in J_{r}$. If $\rk{\alpha_1}>r+1$ or $\rk{\alpha_2}>r+1$, then by \ref{P3}, there exist $\alpha_1',\alpha_2'\in J_{r+1}$ such that $x_{\alpha_1'}x_{\alpha_2'}$ can be obtained from $x_{\alpha_1}x_{\alpha_2}$ using $\R_{r+1,m}$. Clearly $r(\alpha_1\alpha_2)<r(\alpha_1'\alpha_2')$, and $r(w')<r(w)$. 

    Now suppose that $\rk{\alpha_1}=\rk{\alpha_2}=r+1$. Since $r(w)>2r+2$ and $r(\alpha_1\alpha_2)=2r+2$, we have $k\geq 3$. Consider the prefix $x_{\alpha_1}x_{\alpha_2}x_{\alpha_3}$. Since $\rk{\alpha_2\alpha_3}\leq \rk{\alpha_2}=r+1$, we have $\alpha_2\alpha_3\in J_{r}\cup J_{r+1}$. If $\alpha_2\alpha_3\in J_{r+1}$, then $x_{\alpha_2}x_{\alpha_3}=x_{\alpha_2\alpha_3}$ is a relation in $\R_{r+1,m}$. Therefore $r(x_{\alpha_1}x_{\alpha_2\alpha_3}x_{\alpha_4}\dots x_{\alpha_k})<r(w)$. Otherwise, $\alpha_2\alpha_3\in J_{r}$. Then by statement \ref{P2}, $x_{\alpha_1}x_{\alpha_2}x_{\alpha_3}=x_{\alpha_1'}x_{\alpha_3'}$ where $\alpha_1',\alpha_3'\in J_{r+1}$, follows from $\R_{r+1,m}$, and $w'=x_{\alpha_1'}x_{\alpha_3'}\dots x_{\alpha_k}$ satisfies $r(w')<r(w)$.
\end{proof}
(ii) Now suppose $w=x_{\alpha_1}\dots x_{\alpha_k}$ and $v=x_{\beta_1}\dots x_{\beta_l}$ are words in $X_{r+1,m}^*$ such that they represent the same element of rank $r$. Then by the first part of the lemma, we can obtain $x_\alpha x_\beta$ from $w$, and $x_\gamma x_\delta$ from $v$, where $\alpha,\beta,\gamma,\delta\in J_{r+1}$. Then by \ref{P1}, we may deduce the relation $x_\gamma x_\delta=x_\alpha x_\beta$ from $\R_{r+1,m}$.
\end{proof}

The above points out the direction we should aim for when finding the upper bound. In the following, we build up on the lemma and prove that once the conditions $\ref{P1}-\ref{P3}$ are satisfied, the upper bound would be straightforward.

Recall that $\C_{r}=\langle X_{r,m} \mid \R_{r,m} \rangle$ is the restriction of the Cayley table presentation $\C$ of $S$ to the generators of rank at least $r$. Define \begin{center}
    $R_r'=\{u=v\colon u,v\in X_{r+1,m}^+, {\rm and\ } u,v {\rm\ represent\ a\ map\ of\ rank\ } r\}$. 
\end{center}

\begin{prop}\label{prop: All: Cr and Cr+1 define the same semigroup}
     For $\epsilon\leq r\leq 2m-n-1$, if $S$ satisfies $\ref{P1}-\ref{P3}$, then
     $\C_r$ and $\C_{r+1}$ define the same semigroup.
\end{prop}

\begin{proof}
    We show that we can obtain $\C_r$ from $\C_{r+1}$ using Tietze transformations. 
    
    \underline{Step 1}: Add a set of generating symbols $\{x_\alpha\colon \rk\alpha=r\}$ to $\langle X_{r+1,m}\ |\ \R_{r+1,m}\rangle$. For each new symbol $x_\alpha$, choose a word $u_\alpha\in X_{r+1,m}^+$ which represent the map $\alpha$, and add a relation $x_\alpha=u_\alpha$ \ref{T3}. The presentation obtained is $P_1=\langle X_{r,m}\ |\ \R_{r+1,m}, x_\alpha=u_\alpha\ (\alpha\in D_r)\rangle$.
    
    \underline{Step 2}: By Lemma \ref{lem: All: words of rank r are cons of above}, $R_r'$ is a consequence of $\C_{r+1}$. Define a set of relations $R_r''=\{x_\beta x_\gamma=x_{\beta\gamma}\colon \rk\beta,\rk\gamma\geq r, \rk{\beta\gamma}=r\}$. Consider $\beta,\gamma\in J_r\cup \dots\cup J_m$ such that $\beta\gamma\in J_r$. Since $\beta\gamma\in J_r$, $x_{\beta\gamma}=u_{\beta\gamma}$ is a relation in $P_1$, if $\beta\in J_r$ then $x_\beta=u_\beta$ is a relation in $P_1$, otherwise $\rk\beta> r$, and $x_\beta\in X_{r+1,m}$. Similarly, either $x_\gamma=u_\gamma$ is a relation in $P_1$, or $x_\gamma\in X_{r+1,m}$. We consider the $4$ possible combinations of the ranks of $\beta$ and $\gamma$. If $\beta,\gamma\in J_r$, then $x_\beta x_\gamma, u_{\beta\gamma}\in X_{r+1,m}^+$ such that they represent the same rank-$r$ map. Then $x_\beta x_\gamma=u_{\beta\gamma}\in R_r'$ and is a consequence of $\C_{r+1}$ and hence $P_1$. If $\beta,\gamma\in J_r$, since $u_\beta=x_\beta, u_\gamma=x_\gamma$ are relations in $P_1$, and $u_\beta u_\gamma=u_{\beta\gamma}\in R_r'$, we can deduce $x_\beta x_\gamma=x_{\beta\gamma}$ from $P_1$. If exactly one of $\beta$ and $\gamma$ is of rank $r$, say $\beta\in J_r$ and $\gamma\notin J_r$, then $x_\beta=u_\beta$ is a relation in $P_1$. The words $u_\beta x_\gamma,u_{\beta\gamma}\in X_{r+1,m}^+$ represent the same map $\beta\gamma$, so $u_\beta x_\gamma=u_{\beta\gamma}\in R_r'$ and is a consequence of $\C_r$ and hence $P_1$. A similar argument holds for the case when $\beta\notin J_r$ and $\gamma\in J_r$. In all cases $R_r''$ follows from $P_1$. So adding $R_r''$ we obtain $P_2$. Notice that $\R_{r,m}=\R_{r+1,m}\cup\{x_\beta x_\gamma=x_{\beta\gamma}\colon\beta\in J_r {\rm \ or\ }\gamma\in J_r{\ \rm or\ }\beta\gamma\in J_r\}=\R_{r+1,m}\cup R_r''$, so $P_2=\langle X_{r,m}\ |\ \R_{r,m}, x_\alpha=u_\alpha\ (\alpha\in J_r)\rangle$.
    
    \underline{Step 3}: Consider a relation of the form $x_\alpha=u_\alpha$ where $\alpha\in J_r$. We have $u_\alpha\equiv x_{\alpha_1}x_{\alpha_2}\dots x_{\alpha_k}\in X_{r+1,m}^+$. The sequence of relations $x_{\alpha_1 \alpha_2\dots\alpha_k}=x_{\alpha_1}x_{\alpha_2\dots\alpha_k}=\dots=x_{\alpha_1}x_{\alpha_2}\dots x_{\alpha_k}$ is a consequence of $\langle X_{r,m}\ |\ \R_{r,m}\rangle$. Therefore removing $\{x_\alpha=u_\alpha\colon\alpha\in J_r\}$ from $P_2$ we obtain $\C_r$.
    \end{proof}

\begin{thm}\label{thm: All: presentation for Im}
    Suppose that for $\epsilon\leq r\leq 2m-n-1$, $S$ satisfies $\ref{P1}-\ref{P3}$. Let $s=2m-n$. Then $\C_s$ defines a presentation for $I_m$.
\end{thm}

\begin{proof}
    By repeatedly applying Proposition \ref{prop: All: Cr and Cr+1 define the same semigroup}, $C_\epsilon, C_{\epsilon+1},...,C_s$ all define the same semigroup. However, $C_\epsilon$ is the Cayley table of $I_m$, and hence defines $I_m$ itself.
\end{proof}

Given a semigroup $\mathcal{S}$ where $\mathcal{S}$ is one of $\I_n,\T_n,\PT_n$ or their ideals, we have established a way of proving that a given depth is an upper bound for $\dep{\mathcal{S}}$ by proving that \ref{P1}-\ref{P3} hold. In Sections \ref{sec: I_n}, \ref{sec: T_n}, and \ref{sec: PT_n}, we will find out the conditions when \ref{P1}-\ref{P3} are satisfied for $\I_n$, $\T_n$ and $\PT_n$ respectively, therefore obtaining an upper bound for the depth of each of them.

\section{Upper bound for $\I_n$}\label{sec: I_n}
In this section we consider the upper bound for the relational depth of `large' ideals $I_m$ in $\I_n$, where $m>\frac{n}{2}$. This is achieved in Proposition \ref{prop: In: upper bound for In}. Together with Proposition \ref{prop: lower bound for depth of all ideals}, this will yield Theorem \ref{thm: almost main theorem} for $\I_n$.

Suppose that $I_m$ is a proper ideal of $\I_n$, and that $r<m$. Consider partial bijections $\alpha,\beta\in J_{r+1}$ such that $\alpha\beta\in J_r$. Without loss of generality, let   
    \begin{align*}
        \alpha = \begin{small}\begin{pmatrix}
        a_1 & a_2 & \dots & a_r & a_{r+1}\\
        a_1' & a_2' & \dots & a_r' & a_{r+1}'
        \end{pmatrix}\end{small},\ 
        \beta = \begin{small}\begin{pmatrix}
        a_1' & a_2' & \dots & a_r' & b_{r+1}\\
        b_1' & b_2' & \dots & b_r' & b_{r+1}'
        \end{pmatrix}\end{small},
    \end{align*}where $a_{r+1}'\in [n]\setminus\dm{\beta}$. For simplicity, we further define that $A=[n]\setminus\dm{\alpha}$, and $B=[n]\setminus\dm{\beta}$. The following lemmas are examples of words that we can obtain from $x_{\alpha}x_{\beta}$ using $\R_{r+1,r+2}$.

\begin{lemma}\label{lem: In: change im alpha}
    Let $a\in B\setminus\{a_{r+1}'\}$. Define $\alpha' = \begin{small}
        \begin{pmatrix}
        a_1 & a_2 & \dots & a_r & a_{r+1}\\
        a_1' & a_2' & \dots & a_r' & a
    \end{pmatrix}
    \end{small}$. Then $\alpha\beta=\alpha'\beta$, and the relation $x_\alpha x_\beta=x_{\alpha'} x_{\beta}$ is a consequence of $\R_{r+1,r+2}$.
\end{lemma}

\begin{proof}
    Let $b\in [n]\setminus\im\beta$. Define \begin{align*}
        \beta_1&=\begin{small}\begin{pmatrix}
            a_1' & a_2' & \dots & a_r' & b_{r+1} & a_{r+1}'\\
            b_1 & b_2 & \dots & b_r &  b_{r+1} & b
        \end{pmatrix}\end{small},\ 
        \beta_1'=\begin{small}\begin{pmatrix}
            a_1' & a_2' & \dots & a_r' & b_{r+1} & a\\
            b_1 & b_2 & \dots & b_r &  b_{r+1} & b 
        \end{pmatrix}\end{small},\\
        \beta_2&=\begin{small}\begin{pmatrix}
            b_1' & b_2' & \dots & b_r' &  b_{r+1}' \\
            b_1' & b_2' & \dots & b_r' &  b_{r+1}'
        \end{pmatrix}.\end{small}\end{align*}
        Then $\alpha\beta_1=\alpha'\beta_1'\in J_{r+1}$, $\beta_1\beta_2=\beta_1'\beta_2=\beta$. Hence \begin{align*}
            \begin{alignedat}{2}
            x_\alpha x_\beta&=x_{\alpha}x_{\beta_1}x_{\beta_2}\ \ &&(x_\beta=x_{\beta_1}x_{\beta_2})\\&=x_{\alpha'}x_{\beta_1'}x_{\beta_2}&&(x_{\alpha}x_{\beta_1}=x_{\alpha'}x_{\beta_1'}) \\&=x_{\alpha'} x_{\beta}&&(x_{\beta_1'}x_{\beta_2}=x_{\beta}).        
            \end{alignedat}
        \end{align*}
\end{proof}

\begin{lemma}\label{lem: In: change ker alpha}
    Let $a\in A$. Define $\alpha' = \begin{small}
        \begin{pmatrix}
        a_1 & a_2 & \dots & a_r & a \\
        a_1' & a_2' & \dots & a_r' & a_{r+1}' 
    \end{pmatrix}
    \end{small}$. Then $\alpha\beta=\alpha'\beta$, and the relation $x_\alpha x_\beta=x_{\alpha'} x_{\beta}$ is a consequence of $\R_{r+1,r+2}$.
\end{lemma}

\begin{proof}
    Let $a'\in A\setminus\{a\}$, let $b\in B$. Define \begin{align*}
        \alpha_1&=\begin{small}\begin{pmatrix}
            a_1 & a_2 & \dots & a_r & a_{r+1} & a\\
            a_1 & a_2 & \dots & a_r & a_{r+1} & a 
        \end{pmatrix}\end{small},\ 
        \alpha_2=\begin{small}\begin{pmatrix}
            a_1 & a_2 & \dots & a_r & a_{r+1} & a' \\
            a_1' & a_2' & \dots & a_r' & a_{r+1}' & b 
        \end{pmatrix}\end{small},\\
        \alpha_2'&=\begin{small}\begin{pmatrix}
            a_1 & a_2 & \dots & a_r & a & a'\\
            a_1' & a_2' & \dots & a_r' & a_{r+1}' & b
        \end{pmatrix}\end{small},\ 
        \beta_1=\begin{small}\begin{pmatrix}
            a_1' & a_2' & \dots & a_r' & b_{r+1} & b \\
            a_1' & a_2' & \dots & a_r' & b_{r+1} & b 
        \end{pmatrix}\end{small},\\
        \beta_2&=\begin{small}\begin{pmatrix}
            a_1' & a_2' & \dots & a_r' & b_{r+1} \\
            b_1' & b_2' & \dots & b_r' & b_{r+1}' 
        \end{pmatrix}\end{small}.
        \end{align*}We have $\alpha_1\alpha_2=\alpha, \alpha_1\alpha_2'=\alpha', \beta_1\beta_2=\beta, \alpha_2\beta_1=\alpha_2'\beta_1\in J_{r+1}$. Hence \begin{align*}
        \begin{alignedat}{2}
            x_\alpha x_\beta&=x_{\alpha_1}x_{\alpha_2}x_{\beta_1}x_{\beta_2}\ \ &&(x_{\alpha}=x_{\alpha_1}x_{\alpha_2},x_\beta=x_{\beta_1}x_{\beta_2})\\&=x_{\alpha_1}x_{\alpha_2'}x_{\beta_1}x_{\beta_2} \ \ &&(x_{\alpha_2}x_{\beta_1}=x_{\alpha_2'}x_{\beta_1}) \\&=x_{\alpha'} x_{\beta}\ \ &&(x_{\alpha_1}x_{\alpha_2'}=x_{\alpha'}).
        \end{alignedat}
        \end{align*}
\end{proof}

\begin{lemma}\label{lem: In: change im beta}
    Let $b\in [n]\setminus\im\beta$. Define $\beta' = \begin{small}
        \begin{pmatrix}
        a_1' & a_2' & \dots & a_r' & b_{r+1} \\
        b_1' & b_2' & \dots & b_r' & b 
    \end{pmatrix}
    \end{small}$. Then $\alpha\beta=\alpha\beta'$, and the relation $x_\alpha x_\beta=x_{\alpha} x_{\beta'}$ is a consequence of $\R_{r+1,r+2}$.
\end{lemma}

\begin{proof}
    Let $a\in B\setminus\{a_{r+1}'\}, b'\in [n]\setminus(\im\beta\cup\{b\})$. Define \begin{align*}
        \alpha_2&=\begin{small}\begin{pmatrix}
            a_1' & a_2' & \dots & a_r' & a_{r+1}' & a \\
            a_1' & a_2' & \dots & a_r' & a_{r+1}' & a 
        \end{pmatrix}\end{small},\ 
        \beta_1=\begin{small}\begin{pmatrix}
            a_1' & a_2' & \dots & a_r' & b_{r+1} & a \\
            b_1' & b_2' & \dots & b_r' & b_{r+1}' & b' 
        \end{pmatrix}\end{small},\\
        \beta_1'&=\begin{small}\begin{pmatrix}
            a_1' & a_2' & \dots & a_r' & b_{r+1} & a \\
            b_1' & b_2' & \dots & b_r' & b & b' 
        \end{pmatrix}\end{small},\ 
        \beta_2=\begin{small}\begin{pmatrix}
            b_1' & b_2' & \dots & b_r' & b_{r+1}' & b \\
            b_1' & b_2' & \dots & b_r' & b_{r+1}' & b 
        \end{pmatrix}\end{small}.
        \end{align*}Then $\alpha\alpha_2=\alpha, \beta_1\beta_2=\beta, \beta_1'\beta_2=\beta',\alpha_2\beta_1=\alpha_2\beta_1'\in J_{r+1}$. Hence \begin{align*}
            \begin{alignedat}{2}
            x_\alpha x_\beta&=x_{\alpha}x_{\alpha_2}x_{\beta_1}x_{\beta_2}\ \ &&(x_{\alpha}=x_{\alpha}x_{\alpha_2},x_\beta=x_{\beta_1}x_{\beta_2})\\&=x_{\alpha}x_{\alpha_2}x_{\beta_1'}x_{\beta_2} \ \ &&(x_{\alpha_2}x_{\beta_1}=x_{\alpha_2}x_{\beta_1'}) \\&=x_{\alpha} x_{\beta'}\ \ &&(x_{\beta_1'}x_{\beta_2}=x_{\beta'}).
            \end{alignedat}
        \end{align*}
\end{proof}

\begin{lemma}\label{lem: In: local change ker beta}
    Let $b\in B\setminus\{a_{r+1}\}$. Define $\beta' = \begin{small}
        \begin{pmatrix}
        a_1' & a_2' & \dots & a_r' & b \\
        b_1' & b_2' & \dots & b_r' & b_{r+1}' 
    \end{pmatrix}
    \end{small}$. Then $\alpha\beta=\alpha\beta'$, and the relation $x_\alpha x_\beta=x_{\alpha} x_{\beta'}$ is a consequence of $\R_{r+1,r+2}$.
\end{lemma}

\begin{proof}
    Let $a\in A, b'\in [n]\setminus\im\beta$. Define \begin{align*}
        \alpha_1&=\begin{small}\begin{pmatrix}
            a_1 & a_2 & \dots & a_r & a_{r+1} \\
            a_1 & a_2 & \dots & a_r & a_{r+1} 
        \end{pmatrix}\end{small},\ 
        \alpha_2=\begin{small}\begin{pmatrix}
            a_1 & a_2 & \dots & a_r & a_{r+1} & a \\
            a_1' & a_2' & \dots & a_r' & a_{r+1}' & b_{r+1} 
        \end{pmatrix}\end{small},\\
        \alpha_2'&=\begin{small}\begin{pmatrix}
            a_1 & a_2 & \dots & a_r & a_{r+1} & a \\
            a_1' & a_2' & \dots & a_r' & a_{r+1}' & b 
        \end{pmatrix}\end{small},\ 
        \beta_1=\begin{small}\begin{pmatrix}
            a_1' & a_2' & \dots & a_r' & b_{r+1} & b\\
            b_1' & b_2' & \dots & b_r' & b_{r+1}' & b' 
        \end{pmatrix}\end{small},\\
        \beta_1'&=\begin{small}\begin{pmatrix}
            a_1' & a_2' & \dots & a_r' & b_{r+1} & b \\
            b_1' & b_2' & \dots & b_r' & b' & b_{r+1}' 
        \end{pmatrix}\end{small},\
        \beta_2=\begin{small}\begin{pmatrix}
            b_1' & b_2' & \dots & b_r' & b_{r+1}' \\
            b_1' & b_2' & \dots & b_r' & b_{r+1}' 
        \end{pmatrix}\end{small}.
        \end{align*}Then $\alpha_1\alpha_2=\alpha_1\alpha_2'=\alpha, \beta_1\beta_2=\beta,\beta_1'\beta_2=\beta', \alpha_2\beta_1=\alpha_2'\beta_1'\in J_{r+1}$. Hence \begin{align*}
            \begin{alignedat}{2}
            x_\alpha x_\beta&=x_{\alpha_1}x_{\alpha_2}x_{\beta_1}x_{\beta_2}\ \ &&(x_{\alpha}=x_{\alpha_1}x_{\alpha_2},x_\beta=x_{\beta_1}x_{\beta_2})\\&=x_{\alpha_1}x_{\alpha_2'}x_{\beta_1'}x_{\beta_2} \ \ &&(x_{\alpha_2}x_{\beta_1}=x_{\alpha_2'}x_{\beta_1'}) \\&=x_{\alpha} x_{\beta'}\ \ &&(x_{\alpha_1}x_{\alpha_2'}=x_{\alpha}, x_{\beta_1'}x_{\beta_2}=x_{\beta'}).
            \end{alignedat}
        \end{align*}
\end{proof}

\begin{lemma}\label{lem: In: global change ker beta}
    If $|B|\geq 2$, let $a'\in B\setminus\{a_{r+1}'\}$. Let $1\leq i\leq r$. Define \begin{align*}
        \alpha' = \begin{small}
        \begin{pmatrix}
        a_1 & \dots & a_i & \dots & a_{r+1} \\
        a_1' & \dots & a' & \dots & a_{r+1}'
    \end{pmatrix}
    \end{small},\ \beta' = \begin{small}
        \begin{pmatrix}
        a_1' & \dots & a'& \dots & b_{r+1} \\
        b_1' & \dots & b_i' & \dots & b_{r+1}' 
    \end{pmatrix}
    \end{small}.
    \end{align*} Then $\alpha\beta=\alpha'\beta'$, and the relation $x_\alpha x_\beta=x_{\alpha'} x_{\beta'}$ is a consequence of $\R_{r+1,r+2}$.
\end{lemma}

\begin{proof}
    Without loss of generality, let $i=1$. Let $a\in A$. Define \begin{align*}
    \alpha_1 &= \begin{small}
        \begin{pmatrix}
        a_1 & \dots & a_r & a_{r+1} \\
        a_1' & \dots & a_r' & a_{r+1}' 
    \end{pmatrix}
    \end{small},\ \alpha_2 = \begin{small}
        \begin{pmatrix}
        a_1 & \dots & a_r & a_{r+1} & a \\
        a_1' & \dots & a_r' & a_{r+1}' & b_{r+1} 
    \end{pmatrix}\end{small},\\
    \alpha_2' &= \begin{small}
        \begin{pmatrix}
        a_1 & \dots & a_r & a_{r+1} & a \\
        a' & \dots & a_r' & a_{r+1}' & b_{r+1} 
    \end{pmatrix}
    \end{small},\ \beta_1 = \begin{small}
        \begin{pmatrix}
        a_1' & a' & \dots & a_r' & b_{r+1} \\
        a_1' & a' & \dots & a_r' & b_{r+1} 
    \end{pmatrix}
    \end{small},\\ \beta_1' &= \begin{small}
        \begin{pmatrix}
        a_1' & a' & \dots & a_r' & b_{r+1}  \\
        a' & a_1' & \dots & a_r' & b_{r+1} 
    \end{pmatrix}
    \end{small}. 
    \end{align*}Then $\alpha_1\alpha_2=\alpha$, $\alpha_1\alpha_2'=\alpha'$, $\alpha_2\beta_1=\alpha_2'\beta_1'=\begin{small}
        \begin{pmatrix}
            a_1 & \dots & a_r & a \\
            a_1' & \dots & a_r' & b_{r+1} 
        \end{pmatrix}
    \end{small}\in J_{r+1}$, $\beta_1\beta=\beta$, $\beta_1'\beta=\beta'$. Hence \begin{align*} 
    \begin{alignedat}{2}
    x_\alpha x_\beta&=x_{\alpha_1}x_{\alpha_2}x_{\beta_1}x_{\beta}\ \ &&(x_{\alpha}=x_{\alpha_1}x_{\alpha_2},x_\beta=x_{\beta_1}x_{\beta})\\&=x_{\alpha_1}x_{\alpha_2'}x_{\beta_1'}x_{\beta} \ \ &&(x_{\alpha_2}x_{\beta_1}=x_{\alpha_2'}x_{\beta_1'}) \\&=x_{\alpha'} x_{\beta'}\ \ &&(x_{\alpha_1}x_{\alpha_2'}=x_{\alpha'}, x_{\beta_1'}x_{\beta}=x_{\beta'}).
    \end{alignedat}
    \end{align*}
\end{proof}

The above are examples of words that can be obtained from $x_\alpha x_\beta$. As a result, we are allowed to replace certain letters in a given word using consequences of $\R_{r+1,r+2}$.

\begin{lemma}\label{lem: In: all length-2 words obtained from r+1,r+2}
    Let $\alpha,\beta$ be as defined above, and suppose that $\alpha'$, $\beta'$ are obtained from $\alpha,\beta$ using one of the following rules: \begin{thmenumerate}
        \item \label{lem: In: change word: change im alpha} If $a\in B\setminus\{a_{r+1}'\}$, let $\alpha'$ be obtained from $\alpha$ by changing the image of $a_{r+1}$ from $a_{r+1}'$ to $a$, and let $\beta'=\beta$.
        \item \label{lem: In: change word: change ker alpha} If $a\in A$, let $\alpha'$ be obtained from $\alpha$ by replacing the kernel class $a_{r+1}$ with $a$, and let $\beta'=\beta$.
        \item \label{lem: In: change word: change im beta} If $b\in [n]\setminus\im\beta$, let $\beta'$ be obtained from $\beta$ by changing the image of $b_{r+1}$ from $b_{r+1}'$ to $b$, and let $\alpha'=\alpha$.
        \item \label{lem: In: change word: change ker beta locally} If $b\in B\setminus\{a_{r+1}\}$, let $\beta'$ be obtained from $\beta$ by replacing the kernel class $b_{r+1}$ with $b$, and let $\alpha'=\alpha$.
        \item \label{lem: In: change word: change ker beta globally} If $|B|\geq 2$, $1\leq i\leq r$, and $a'\in B\setminus\{a_{r+1}'\}$, let $\alpha'$ be obtained form $\alpha$ by changing the image of $a_i$ from $a_i'$ to $a'$, and let $\beta'$ be obtained from $\beta$ by replacing the kernel class $a_i'$ with $a'$.
    \end{thmenumerate}Then the word $x_{\alpha'} x_{\beta'}$ can be obtained from $x_\alpha x_\beta$ using $\R_{r+1,r+2}$.
\end{lemma}

\begin{proof}
    This follows from Lemmas \ref{lem: In: change im alpha} - \ref{lem: In: global change ker beta}.
\end{proof}

Now, we work towards showing that ideals of $\I_n$ satisfy the conditions $\ref{P1}-\ref{P3}$. We start by showing that given a word of rank $r$ which is of length 2, we can obtain any word of length 2 that represents the same map, by applying the `rewriting' rules given in Lemma \ref{lem: In: all length-2 words obtained from r+1,r+2}.
\begin{lemma}\label{lem: In: length 2 words are equal}
    Let $\alpha,\beta,\gamma,\delta\in J_{r+1}$ such that $\alpha\beta=\gamma\delta\in J_r$. Then $x_\alpha x_\beta=x_\gamma x_\delta$ is a consequence of $\R_{r+1,r+2}$.
\end{lemma}
\begin{proof}
    Let \[\alpha\beta=\gamma\delta=\begin{small}
        \begin{pmatrix}
            a_1 & a_2 & \dots & a_r\\
            b_1' & b_2' & \dots & b_r' 
        \end{pmatrix}
    \end{small}.\] Then without loss of generality\[\alpha=\begin{small}
        \begin{pmatrix}
            a_1 & \dots & a_r & a_{r+1} \\
            a_1' & \dots & a_r' & a_{r+1}' 
        \end{pmatrix}
    \end{small},\ \beta=\begin{small}
        \begin{pmatrix}
            a_1' & \dots & a_r' & b_{r+1} \\
            b_1' & \dots & b_r' & b_{r+1}' 
        \end{pmatrix}
    \end{small},\] \[\gamma=\begin{small}
        \begin{pmatrix}
            a_1 & \dots & a_r & c_{r+1} \\
            c_1' & \dots & c_r' & c_{r+1}' 
        \end{pmatrix}
    \end{small},\ \delta=\begin{small}
        \begin{pmatrix}
            c_1' & \dots & c_r' & d_{r+1} \\
            b_1' & \dots & b_r' & d_{r+1}'
        \end{pmatrix}
    \end{small}.\] Let $A, B, C, D$ denote the complements of the mappings $\alpha, \beta, \gamma, \delta$ respectively. We aim to show that we can obtain $x_\alpha x_\beta$ from $x_\gamma x_\delta$ using the rules in Lemma \ref{lem: In: all length-2 words obtained from r+1,r+2}. We begin by letting $i=1$.
    
    \underline{Step 1}: If $c_i'=a_i'$, go to Step 2. Otherwise, $c_i'\neq a_i'$. If $a_i'\in D$, change $\gamma$ and $\delta$ by changing the image of $a_i$ and preimage of $b_i'$ from $c_i'$ to $a_i'$ (5), respectively. If $a_i'=d_{r+1}$, first change $\delta$ by replacing the preimage of $d_{r+1}$ with some $d\in D$, then change the image of $a_i$ and preimage of $b_i'$ to $a_i'$. Similarly, if $a_i'=c_j'$ for some $j\neq i$, we change the image of $a_i$ and preimage of $b_i'$ using \ref{lem: In: change word: change ker beta globally}.

    \underline{Step 2}: If $i=r$, go to Step 3. Otherwise, let $i=i+1$.
    At this point, we should have \[\gamma=\begin{small}
        \begin{pmatrix}
            a_1 & \dots & a_r & c_{r+1} \\
            a_1' & \dots & a_r' & c_{r+1}' 
        \end{pmatrix}
    \end{small}, \text{ and } \delta=\begin{small}
        \begin{pmatrix}
            a_1' & \dots & a_r' & d \\
            b_1' & \dots & b_r' & d_{r+1}' 
        \end{pmatrix}
    \end{small}.\]
    
    \underline{Step 3}: Change $\gamma$ by first changing the kernel class $c_{r+1}$ to $a_{r+1}$ \ref{lem: In: change word: change ker alpha}, then the image of $a_{r+1}$ from $c_{r+1}'$ to $a_{r+1}'$ \ref{lem: In: change word: change im alpha}, obtaining $\alpha$. Similarly, we may change $\delta$ using rules \ref{lem: In: change word: change im beta} and \ref{lem: In: change word: change ker beta locally}, obtaining $\beta$.
\end{proof}

\begin{lemma}\label{lem: In: length 3 and length 2 words}
    Let $0\leq r\leq 2m-n-1$. If $\alpha, \beta, \gamma\in J_{r+1}, \alpha\beta, \beta\gamma, \alpha\beta\gamma\in J_{r}$, then there is a partial bijection $\alpha'\in J_{r+1}$ such that $\alpha\beta\gamma=\alpha'\gamma$, and the relation $x_\alpha x_\beta x_\gamma=x_{\alpha'} x_\gamma$ is a consequence of $\R_{r+1,r+2}$.
\end{lemma}
\begin{proof}
    Suppose that \[\alpha = \begin{small}\begin{pmatrix}
        a_1 & a_2 & \dots & a_r & a_{r+1}\\
        a_1' & a_2' & \dots & a_r' & a_{r+1}' 
        \end{pmatrix}\end{small},
        \beta = \begin{small}\begin{pmatrix}
        a_1' & a_2' & \dots & a_r' & b_{r+1} \\
        b_1' & b_2' & \dots & b_r' & b_{r+1}' 
        \end{pmatrix}\end{small}.\] Since $\alpha\beta\gamma=\beta\gamma\in J_r$, then $\im{\alpha\beta}\subseteq\dm{\gamma}$, and \[\gamma=\begin{small}
            \begin{pmatrix}
                b_1' & \dots & b_r' & c_{r+1} \\
                c_1' & \dots & c_r' & c_{r+1}'
            \end{pmatrix}
        \end{small}.\] Define \begin{align*}
            \alpha'=\begin{small}
            \begin{pmatrix}
                a_1 & \dots & a_r & a_{r+1} \\
                b_1' & \dots & b_r' & b_{r+1}
            \end{pmatrix}
        \end{small}, \beta'=\begin{small}
            \begin{pmatrix}
                a_1' & \dots & a_r' & b_{r+1} \\
                b_1' & \dots & b_r' & c_{r+1} 
            \end{pmatrix}
        \end{small}.            
        \end{align*} Then \begin{align*}
            \alpha'\gamma=\begin{small}
            \begin{pmatrix}
                a_1 & \dots & a_r \\
                c_1' & \dots & c_r'
            \end{pmatrix}
        \end{small}=\alpha\beta\gamma, \beta'\gamma=\begin{small}
            \begin{pmatrix}
                a_1' & \dots & a_r' & b_{r+1} \\
                c_1' & \dots & c_r' & c_{r+1} 
            \end{pmatrix}
        \end{small}\in J_{r+1},
        \end{align*} and $\alpha\beta'=\alpha\beta$. Hence \begin{align*}
        \begin{alignedat}{2}
            x_\alpha x_\beta x_\gamma&=x_{\alpha}x_{\beta'}x_{\gamma}\ \ &&(x_\alpha x_\beta=x_\alpha x_{\beta'})\\&=x_{\alpha}x_{\beta'\gamma} \ \ &&(x_{\beta'}x_{\gamma}=x_{\beta'\gamma} {\rm\  is\ a\ relation\ in}\ \R_{r+1,r+2}) \\&=x_{\alpha'} x_{\gamma}\ \ &&({\rm Lemma} \ \ref{lem: In: length 2 words are equal}).
            \end{alignedat}
        \end{align*}
\end{proof}

\begin{lemma}\label{lem: In: can obtain words with higher rank letter}
    Let $0\leq r\leq 2m-n-1$. If $\gamma\delta\in J_{r}$ where $\rk\gamma,\rk\delta\geq r+1$, there exist $\alpha,\beta\in J_{r+1}$ such that $\alpha\beta=\gamma\delta$, and $x_\alpha x_\beta=x_\gamma x_\delta$ is a consequence of $\R_{r+1,m}$.
\end{lemma}
\begin{proof}
    Consider \[\gamma=\begin{small}
        \begin{pmatrix}
            c_1 & \dots & c_i \\
            c_1' & \dots & c_i'
        \end{pmatrix}
    \end{small}\in J_i, \text{ and } \delta=\begin{small}
        \begin{pmatrix}
            d_1 & \dots & d_j \\
            d_1' & \dots & d_j' 
        \end{pmatrix}
    \end{small}\in J_j,\] where $i,j\geq r+1$. Since $\gamma\delta\in J_r$, $r$ elements of $\im\gamma$ are in $\dm{\delta}$, without loss of generality, let $d_k=c_k'$ for $1\leq k\leq r$. Then \begin{align*}
        \delta=\begin{small}
        \begin{pmatrix}
            c_1' & \dots & c_r' & d_{r+1} & \dots & d_j \\
            d_1' & \dots & d_r' & d_{r+1}' & \dots & d_j' 
        \end{pmatrix}\end{small},\text{ and } \gamma\delta=\begin{small}
            \begin{pmatrix}
                c_1 & \dots & c_r  \\
                d_1' & \dots & d_r' 
            \end{pmatrix}
        \end{small}.
    \end{align*}
        If $j<m$, define \begin{align*}
            \delta_1=\begin{small}
            \begin{pmatrix}
                c_1' & \dots & c_r' & c_{r+1}' & d_{r+1} & \dots & d_j\\
                c_1' & \dots & c_r' & c_{r+1}' & d_{r+1} & \dots & d_j
            \end{pmatrix}
        \end{small}\in J_{j+1}, \alpha=\begin{small}
            \begin{pmatrix}
                c_1 & \dots & c_r & c_{r+1} \\
                c_1' & \dots & c_r' & c_{r+1}' 
            \end{pmatrix}
        \end{small}\in J_{r+1}.
        \end{align*} Then $\delta_1\delta=\delta$, and $\gamma\delta_1=\alpha$. Hence \begin{align*}
        \begin{alignedat}{2}
        x_\gamma x_\delta&= x_\gamma x_{\delta_1}x_\delta \ \ &&(x_\delta=x_{\delta_1} x_{\delta} {\rm\ is\ a\ relation\ in\ } \R_{r+1,m}),\\
        &=x_\alpha x_\delta \ \ &&(x_\alpha=x_{\gamma} x_{\delta_1}\ {\rm\ is\ a\ relation\ in\ } \R_{r+1,m}).
        \end{alignedat}
    \end{align*}
    If $i<m$, define \begin{align*}
        \gamma_1=\begin{small}
        \begin{pmatrix}
            c_1' & \dots & c_i' & d_{r+1}\\
            c_1' & \dots & c_i' & d_{r+1} 
        \end{pmatrix}
    \end{small}\in J_{i+1}, \beta=\begin{small}
        \begin{pmatrix}
            c_1' & \dots & c_r' & d_{r+1} \\
            d_1' & \dots & d_r' & d_{r+1}' 
        \end{pmatrix}
    \end{small}\in J_{r+1}.
    \end{align*} Then $\gamma\gamma_1=\gamma$, and $\gamma_1\delta=\beta$. Hence \begin{align*}
        \begin{alignedat}{2}
        x_\gamma x_\delta&= x_\gamma x_{\gamma_1}x_\delta \ \ &&(x_\gamma=x_{\gamma} x_{\gamma_1} {\rm\ is\ a\ relation\ in\ } \R_{r+1,m}),\\
        &=x_\gamma x_\beta \ \ &&(x_{\gamma_1}x_\delta=x_\beta\ {\rm\ is\ a\ relation\ in\ } \R_{r+1,m}).
        \end{alignedat}
    \end{align*}Finally, notice that the condition $r\leq 2m-n-1$ guarantees that we cannot have $i=j=m$. Hence, if $i<m$, we may obtain $x_\gamma x_\beta$ from $x_\gamma x_\delta$, then obtain $x_\alpha x_\beta$ from $x_\gamma x_\beta$, where $\alpha,\beta\in J_{r+1}$. If $i=m$, then $j<m$, the rest is similar.
\end{proof}

\begin{prop}\label{prop: In: upper bound for In}
    Consider $I_m$ an ideal of $\I_n$. If $n>m>\frac{n}{2}$, then $\dep{I_m}\leq n-m+1$.
\end{prop}
\begin{proof}
    We have proved in Lemmas \ref{lem: In: length 2 words are equal}-\ref{lem: In: can obtain words with higher rank letter} that when $m>\frac{n}{2}$, $\ref{P1}-\ref{P3}$ hold. It follows from Theorem \ref{thm: All: presentation for Im} that $\dep{I_m}\leq m-(2m-n)+1=n-m+1$.
\end{proof}

\begin{thm}\label{thm: In: relational depth}
    Let $I_m$ be an ideal of $\I_n$. Then $\dep{I_m}=\begin{cases}
        3 & m=n\\
        n-m+1 & n>m>\frac{n}{2}\\
        m+1 & m\leq \frac{n}{2}
    \end{cases}$.
\end{thm}
\begin{proof}
    When $m\leq\frac{n}{2}$, the exact value for $\dep{I_m}$ is given in Corollary \ref{cor: All: exact value for depth of small ideals}. When $n>m>\frac{n}{2}$, the lower bound for $\dep{I_m}$ is found in Proposition \ref{prop: lower bound for depth of all ideals}, which is shown to be an upper bound in Proposition \ref{prop: In: upper bound for In}. 
    
    Finally, we know that $\dep{\I_n}\leq 3$, as $\dep{I_{n-1}}=2$. Consider partial bijections \begin{align*}
        \alpha=\begin{small}
        \begin{pmatrix}
            1 & 2 & \dots & n-1\\
            1 & 2 & \dots & n-1
        \end{pmatrix}
    \end{small}, \beta=\begin{small}
        \begin{pmatrix}
            1 & \dots & n-2 & n\\
            1 & \dots & n-2 & n-1
        \end{pmatrix}
    \end{small}. \end{align*}Then \[\alpha\beta^2=\alpha\beta=\beta^2=\begin{small}
        \begin{pmatrix}
            1 & 2 & ... & n-2\\
            1 & 2 & ... & n-2
        \end{pmatrix}
    \end{small}.\] However, we cannot obtain $x_\alpha{x_\beta}^2$ from ${x_\beta}^2$ using relations from $\R_{n-1,n}$, as $\alpha, \beta$ satisfy the condition in Proposition \ref{prop: general: an example of relations that cannot be deduced}. Hence $\C_{n-1}=\langle X_{n-1,n}\ |\ \R_{n-1,n}\rangle$ does not define a presentation for $\I_n$, and $\dep{\I_n}>2$, which gives $\dep{\I_n}=3$.
\end{proof}

\section{Upper bound for $\T_n$}
\label{sec: T_n}
We now move to the ideals of the full transformation monoid $\T_n$. We follow the same trajectory as in the previous section, but technical details are different due to the different nature of the elements of $\T_n$ compared to $\I_n$.

To begin with, we find some relations that can be deduced from $\R_{r,m}$ in an ideal $I_m$ of $\T_n$. Let $1\leq r<m-1$. 
Unless stated otherwise, the definitions below will be used throughout this section. Let $\alpha,\beta\in J_{r+1}$ such that $\alpha\beta\in J_r$. Without loss of generality, assume that 
    \begin{align*}
    \alpha &= \begin{small}\begin{pmatrix}
        A_1 & A_2 & \dots & A_r & A_{r+1}\\
        a_1 & a_2 & \dots & a_r & a_{r+1}
        \end{pmatrix}\end{small},\ 
        \beta = \begin{small}\begin{pmatrix}
        B_1 & B_2 & \dots & B_r & B_{r+1}\\
        b_1 & b_2 & \dots & b_r & b_{r+1}
        \end{pmatrix}\end{small},\\
     \alpha\beta&=\begin{small}\begin{pmatrix}
        A_1 & A_2 & \dots &A_{r-1} & A_r\cup A_{r+1}\\b_1 & b_2 & \dots & b_{r-1} & b_r
        \end{pmatrix}\end{small}.
\end{align*}

    More specifically, \begin{align*}
        a_i\in \begin{small}
        \begin{cases}
        B_i & (1\leq i\leq r-1)\\B_r & (i=r,r+1)
    \end{cases}
    \end{small},\text{ and } B_{r+1}\cap \im\alpha=\emptyset.
    \end{align*}We investigate words that can be obtained from $x_\alpha x_\beta$ using relations in $\R_{r+1,r+2}$.

\begin{lemma}\label{lem: Tn: change im beta}
    Let $b_{r+1}'\in [n]\setminus\im\beta$. Define \[\beta' = \begin{small}
        \begin{pmatrix}
        B_1 & B_2 & \dots & B_r & B_{r+1}\\
        b_1 & b_2 & \dots & b_r & b_{r+1}'
    \end{pmatrix}
    \end{small}.\] Then $\alpha\beta=\alpha\beta'$, and the relation $x_\alpha x_\beta=x_\alpha x_{\beta'}$ is a consequence of $\R_{r+1,r+2}$.
\end{lemma}

\begin{proof}
   Define the following transformations,
       \begin{align*}
        \alpha_1 &= \begin{small}\begin{pmatrix}
        a_1 & a_2 & \dots & a_{r-1} & a_{r} & [n]\setminus\{a_1,\dots,a_r\}\\
        a_1 & a_2 & \dots & a_{r-1} & a_r & a_{r+1}
        \end{pmatrix}\end{small},\\
        \beta_1 &= \begin{small}\begin{pmatrix}
        B_1 & B_2 &\dots & B_{r-1} & a_{r} & B_r\setminus\{a_r\} & B_{r+1}\\
        b_1 & b_2 &\dots & b_{r-1} & b_r & b_r' & b_{r+1}\end{pmatrix}\end{small},\\
        \beta_1' &= \begin{small}\begin{pmatrix}
        B_1 & B_2 &\dots & B_{r-1} & a_{r} & B_r\setminus\{a_r\} & B_{r+1}\\
        b_1 & b_2 &\dots & b_{r-1} & b_r & b_r' & b_{r+1}'\end{pmatrix}\end{small},\\
        \text{ and one}&\text{ more mapping }\beta_2\text{ by}\\
        x\beta_2 &= \begin{small}\begin{cases}
            x & (x\in \{b_2, b_3,\dots,b_{r-1}, b_{r+1}, b_{r+1}'\})\\
            b_r & (x=b_r\text{ or }b_r')\\
            b_1 & \text{otherwise}
        \end{cases}\end{small}.
    \end{align*}

    Then $\alpha\alpha_1=\alpha$, $\beta_1\beta_2=\beta$, $\beta_1'\beta_2=\beta'$, and $\alpha_1\beta_1=\alpha_1\beta_1'\in J_{r+1}$. 
    Hence, \begin{align*}
            \begin{alignedat}{2}
            x_\alpha x_\beta&=x_{\alpha}x_{\alpha_1}x_{\beta_1}x_{\beta_2}\ &&(x_\alpha=x_\alpha x_{\alpha_1}, x_\beta=x_{\beta_1}x_{\beta_2})\\&=x_{\alpha}x_{\alpha_1}x_{\beta_1'}x_{\beta_2} \ &&(x_{\alpha_1}x_{\beta_1}=x_{\alpha_1}x_{\beta_1'})\\&=x_\alpha x_{\beta'}\ &&(x_{\beta'}=x_{\beta_1'}x_{\beta_2}).
            \end{alignedat}
        \end{align*}
\end{proof}

\begin{lemma}\label{lem: change ker beta}
    If $|B_{r+1}|\geq 2$, let $b\in B_{r+1}$. Let $1\leq i\leq r$. Define \[\beta' = \begin{small}
        \begin{pmatrix}
        B_1 & B_2 &\dots &B_i\cup \{b\}&\dots & B_{r} & B_{r+1}\setminus\{b\}\\
        b_1 & b_2 &\dots  &b_i &\dots& b_{r} & b_{r+1} \end{pmatrix}
    \end{small}.\] Then $\alpha\beta=\alpha\beta'$, and the relation $x_\alpha x_\beta=x_\alpha x_{\beta'}$ is a consequence of $\R_{r+1,r+2}$. 
\end{lemma}

\begin{proof}
    Here we only consider the case when $i=1$, since the calculation is very similar when $i\in \{2,\dots,r\}$. Let $b'\in B_{r+1}\setminus\{b\}$, and $b_{r+2},b_{r+3}\in [n]\setminus\im\beta$. Define the following transformations,
        \begin{align*}
        \alpha_1 &= \begin{small}\begin{pmatrix}
        a_1 & a_2 & \dots & a_{r} & a_{r+1} & [n]\setminus\{a_1,\dots, a_{r+1}\}\\
        a_1 & a_2 & \dots & a_{r} & a_{r+1} & b'
        \end{pmatrix}\end{small},\\
        \beta_1 &= \begin{small}\begin{pmatrix}
        B_1 & B_2 &\dots & B_{r} & B_{r+1}\setminus\{b\} & b\\
        b_1 & b_2 &\dots & b_{r} & b_{r+1} & b_{r+2}\end{pmatrix}\end{small},\\
        \beta_1' &= \begin{small}\begin{pmatrix}
        B_1 & B_2 &\dots & B_{r} & B_{r+1}\setminus\{b\} & b\\
        b_1 & b_2 &\dots & b_{r} & b_{r+1} & b_{r+3}\end{pmatrix}\end{small},\\
        \beta_2 &= \begin{small}\begin{pmatrix}
        \{b_1,b_{r+3}\} & b_2 & \dots & b_r & [n]\setminus\{b_1,\dots, b_r, b_{r+3}\}\\
        b_1 & b_2 & \dots & b_{r} & b_{r+1}\end{pmatrix}\end{small}.
    \end{align*}
    Then $\alpha\alpha_1=\alpha$, $\beta_1\beta_2=\beta$, $\beta_1'\beta_2=\beta'$, and $\alpha_1\beta_1=\alpha_1\beta_1'$. Notice that $\alpha_1,\beta_1,\beta_1'\in J_{r+2}$, $\beta_2\in J_{r+1}$, and \[\alpha_1\beta_1 = \begin{small}
        \begin{pmatrix}
        a_1 & a_2 & \dots & a_{r-1} & \{a_r, a_{r+1}\} & [n]\setminus\{a_1,\dots, a_{r+1}\}\\
        b_1 & b_2 & \dots & b_{r-1} & b_r & b_{r+1}
        \end{pmatrix}
    \end{small}\in J_{r+1}.\] Hence \begin{align*}
            \begin{alignedat}{2}
            x_\alpha x_\beta&=x_{\alpha}x_{\alpha_1}x_{\beta_1}x_{\beta_2}\ \ &&(x_\alpha=x_{\alpha}x_{\alpha_1},x_\beta=x_{\beta_1}x_{\beta_2})\\&=x_{\alpha}x_{\alpha_1}x_{\beta_1'}x_{\beta_2}\ \ &&(x_{\alpha_1}x_{\beta_1}=x_{\alpha_1}x_{\beta_1'}) \\&=x_\alpha x_{\beta'}\ \ &&(x_{\beta'}=x_{\beta_1'}x_{\beta_2}).
            \end{alignedat}
        \end{align*}
\end{proof}

\begin{lemma}\label{lem: Tn: change ker beta full}
    For $1\leq i\neq j\leq r+1$, if $|B_i\setminus \im\alpha|\geq 1$, let $b\in B_i\setminus \im\alpha$, define \[\beta_{ij}' = \begin{small}
        \begin{pmatrix}
        B_1 & B_2 &\dots &B_i\setminus\{b\}& \dots& B_j\cup\{b\}&\dots & B_{r} & B_{r+1}\\
        b_1 & b_2 &\dots  &b_i &\dots& b_j & \dots& b_{r} & b_{r+1} \end{pmatrix}
    \end{small}.\] Then $\alpha\beta=\alpha\beta_{ij}'$, and the relation $x_\alpha x_\beta=x_\alpha x_{\beta_{ij}'}$ is a consequence of $\R_{r+1,r+2}$. 
\end{lemma}

\begin{proof}
    Define \begin{align*}
        \beta_1 &= \begin{small}
        \begin{pmatrix}
        B_1 & B_2 &\dots &B_i\setminus\{b\}& \dots& B_j&\dots & B_{r} & B_{r+1}\cup \{b\}\\
        b_1 & b_2 &\dots  &b_i & \dots&b_j&\dots& b_{r} & b_{r+1} \end{pmatrix}.
    \end{small}
    \end{align*}By Lemma \ref{lem: change ker beta}, $x_\alpha x_\beta=x_\alpha x_{\beta_1}$ and $x_\alpha x_{\beta_1}=x_\alpha x_{\beta_{ij}'}$ are consequences of $\R_{r+1,r+2}$, and we can obtain $x_\alpha x_{\beta_{ij}'}$ from $x_\alpha x_{\beta}$.
\end{proof}

\begin{lemma}\label{lem: Tn: change ker alpha 1}
    Suppose that $|A_r|\geq 2$, and let $a\in A_r$. Define \[\alpha' = \begin{small}
        \begin{pmatrix}
        A_1 & A_2 & \dots & A_r\setminus\{a\} & A_{r+1}\cup \{a\}\\
        a_1 & a_2 & \dots & a_r & a_{r+1}
    \end{pmatrix}
    \end{small}.\] Then $\alpha\beta=\alpha'\beta$, and the relation $x_\alpha x_\beta=x_{\alpha'} x_\beta$ is a consequence of $\R_{r+1,r+2}$.
\end{lemma}

\begin{proof}
    We split the proof into case, depending on the size of $B_r$. Notice that since both $a_{r},a_{r+1}\in B_r$, $|B_r|\geq 2$.
    
    \underline{Case 1}: $|B_r|\geq 3$. 
    Let $a_{r+2}\in B_r\setminus\{a_r,a_{r+1}\}$. Let $a_{r+3}\in B_{r+1}$. Let $B_{r1},B_{r2}$ partition $B_r$ such that $a_r, a_{r+1}\in B_{r1}$, and $a_{r+2}\in B_{r2}$. Let $B_{r1}',B_{r2}'$ partition $B_r$ such that $a_r\in B_{r1}'$, and $a_{r+1}, a_{r+2}\in B_{r2}'$. Define 
        \begin{align*}
        \alpha_1&=\begin{small}\begin{pmatrix}
        A_{1} & A_2 & \dots& A_{r-1} & A_r\setminus\{a\} & a & A_{r+1}\\
        a_1 & a_2 & \dots & a_{r-1}& a_r & a_{r+1} & a_{r+2}
        \end{pmatrix}\end{small},\\
        \alpha_2 &= \begin{small}\begin{pmatrix}
        B_1 & B_2 & \dots & B_{r-1} & B_{r1} & B_{r2} & B_{r+1}\\
        a_1 & a_2 & \dots & a_{r-1} & a_{r} & a_{r+1} & a_{r+3}
        \end{pmatrix}\end{small},\\
        \beta_1 &= \begin{small}\begin{pmatrix}
        B_1 & B_2 & \dots & B_{r-1} & B_{r1}' & B_{r2}' & B_{r+1}\\
        a_1 & a_2 & \dots & a_{r-1} & a_{r} & a_{r+1} & a_{r+3}\end{pmatrix}\end{small}.
    \end{align*}
    Then $\alpha_1\alpha_2=\alpha$, $\alpha_2\beta=\beta_1\beta=\beta$, and $\alpha_1\beta_1=\alpha'$. We have \begin{align*}
            \begin{alignedat}{2}
            x_\alpha x_\beta&=x_{\alpha_1}x_{\alpha_2}x_{\beta}\ \ &&(x_\alpha=x_{\alpha_1}x_{\alpha_2})\\&=x_{\alpha_1}x_\beta \ \ &&(x_{\alpha_2}x_{\beta}=x_{\beta}) \\&=x_{\alpha_1}x_{\beta_1}x_\beta \ \ &&(x_{\beta}=x_{\beta_1}x_{\beta}) \\&=x_{\alpha'}x_{\beta}\ \ &&(x_{\alpha'}=x_{\alpha_1}x_{\beta_1}).    
            \end{alignedat}
        \end{align*}
    \underline{Case 2}: $|B_r|= 2$.
    By Lemma \ref{lem: Tn: change ker beta full}, if \[\beta'=\begin{small}
        \begin{pmatrix}
        B_1 & B_2 &\dots &B_i\setminus\{b\}&\dots & B_{r}\cup \{b\} & B_{r+1}\\
        b_1 & b_2 &\dots  &b_i&\dots& b_{r} & b_{r+1} 
    \end{pmatrix}
    \end{small},\] where $b\in B_i\setminus\im\alpha$ and $|B_i\setminus\im\alpha|\geq 1$, then $x_\alpha x_\beta=x_\alpha x_{\beta'}$ and $x_{\alpha'}x_\beta=x_{\alpha'} x_{\beta'}$ are consequences of $\R_{r+1,r+2}$. Now apply the method in Case 1 to $x_\alpha x_{\beta'}$, and we have $x_\alpha x_\beta=x_\alpha x_{\beta'}=x_{\alpha'} x_{\beta'}=x_{\alpha'} x_{\beta}$.
\end{proof}

\begin{lemma}\label{lem: Tn: change ker alpha 2}
    Let $1\leq i\leq r-1$. Suppose that $|A_i|\geq 2$ and $|B_i|\geq 2$. Let $A_{i1},A_{i2}$ by any non-empty sets partitioning $A_i$, and let $a_{i1},a_{i2}\in B_i$ be two arbitrary distinct elements. Define \[\alpha'=\begin{small}\begin{pmatrix}
        A_1 & A_2 & \dots & A_{i1} & A_{i2} &\dots  & A_r\cup A_{r+1}\\
        a_1 & a_2 & \dots &a_{i1} & a_{i2}&\dots & a_r
        \end{pmatrix}\end{small}.\] Then $\alpha\beta=\alpha'\beta$, and the relation $x_\alpha x_\beta=x_{\alpha'} x_\beta$ is a consequence of $\R_{r+1,r+2}$. 
\end{lemma}

\begin{proof}
    Without loss of generality, let $i=1$. Let $a\in B_{r+1}$. Define 
        \begin{align*}
        \alpha_1&=\begin{small}\begin{pmatrix}
        A_{11} & A_{12} & A_2 & \dots & A_{r-1} & A_r & A_{r+1}\\
        a_{11} & a_{12} & a_2 & \dots & a_{r-1} & a_r & a_{r+1}
        \end{pmatrix}\end{small},\\
        \alpha_2 &= \begin{small}\begin{pmatrix}
        \{a_{11}, a_{12}\} & a_2 & \dots &a_{r-1} & a_{r} & a_{r+1} & [n]\setminus\{a_{11}, a_{12}, a_2,\dots,a_{r+1}\}\\
        a_1 & a_2 & \dots &a_{r-1} & a_{r} & a_{r+1} & a
        \end{pmatrix}\end{small},\\
        \alpha_2' &= \begin{small}\begin{pmatrix}
        a_{11} & a_{12} & a_2 & \dots & a_{r-1} & \{a_{r}, a_{r+1}\} & [n]\setminus\{a_{11}, a_{12}, a_2,\dots,a_{r+1}\}\\
        a_{11} & a_{12}  &a_2 & \dots &a_{r-1} & a_{r} & a
        \end{pmatrix}\end{small}.\end{align*}
    Then $\alpha_1\alpha_2=\alpha$, $\alpha_1\alpha_2'=\alpha'$, \[\alpha_2\beta=\alpha_2'\beta=\begin{small}
        \begin{pmatrix}
            \{a_{11},a_{12}\} & \dots & \{a_r, a_{r+1}\} & [n]\setminus\{a_{11}, a_{12}, a_2,\dots,a_{r+1}\}\\
            b_1 & \dots & b_r & b_{r+1}
        \end{pmatrix}
    \end{small}\in J_{r+1},\] and $\alpha_1,\alpha_2,\alpha_2'\in J_{r+2}$. Hence, \begin{align*}
            \begin{alignedat}{2}
            x_\alpha x_\beta&=x_{\alpha_1}x_{\alpha_2}x_{\beta}\ \ &&(x_\alpha=x_{\alpha_1}x_{\alpha_2})\\&=x_{\alpha_1}x_{\alpha_2'}x_\beta \ \ &&(x_{\alpha_2}x_{\beta}=x_{\alpha_2'}x_{\beta}) \\&=x_{\alpha'}x_{\beta}\ \ &&(x_{\alpha'}=x_{\alpha_1}x_{\alpha_2'}).
            \end{alignedat}
        \end{align*}
\end{proof}

\begin{lemma}\label{lem: Tn: change im alpha}
    Let $1\leq i\leq r+1$. If $|B_i\setminus(B_i\cap \im\alpha)|\geq 1$, let $a_i'\in B_i\setminus \im\alpha$. Define \[\alpha' = \begin{small}
        \begin{pmatrix}
        A_1 & A_2 & \dots & A_i & \dots & A_r & A_{r+1}\\
        a_1 & a_2 & \dots & a_i' & \dots & a_r & a_{r+1}
    \end{pmatrix}
    \end{small}.\] Then $\alpha\beta=\alpha'\beta$, and the relation $x_\alpha x_\beta=x_{\alpha'} x_\beta$ is a consequence of $\R_{r+1,r+2}$.
\end{lemma}

\begin{proof}
    Let $a\in B_{r+1}$. Define \begin{align*}
        \alpha_1&=\begin{small}
        \begin{pmatrix}
            A_1 & \dots & A_i & \dots & A_r & A_{r+1}\\
            a_1 & \dots & a_i & \dots & a_r & a_{r+1}
        \end{pmatrix}
    \end{small},\ \alpha_2=\begin{small}
        \begin{pmatrix}
            a_1 & \dots & a_i & \dots & a_r & a_{r+1} & [n]\setminus\im{\alpha}\\
            a_1 & \dots & a_i & \dots & a_r & a_{r+1} & a
        \end{pmatrix}
    \end{small},\\ \alpha_2'&=\begin{small}
        \begin{pmatrix}
            a_1 & \dots & a_i & \dots & a_r & a_{r+1} & [n]\setminus{\im{\alpha}}\\
            a_1 & \dots & a_i' & \dots & a_r & a_{r+1} & a
        \end{pmatrix}
    \end{small}. 
    \end{align*}Then $\alpha_1\alpha_2=\alpha,\alpha_1\alpha_2'=\alpha'$, \[\alpha_2\beta=\alpha_2'\beta=\begin{small}
        \begin{pmatrix}
            a_1 & \dots & a_i & \dots & a_{r-1} & \{a_r,a_{r+1}\} & [n]\setminus{\im{\alpha}}\\
            b_1 & \dots & b_i & \dots &b_{r-1} & b_r & b_{r+1}
        \end{pmatrix}
    \end{small}\in J_{r+1}.\] Hence, \begin{align*}
            \begin{alignedat}{2}
             x_\alpha x_\beta&=x_{\alpha_1}x_{\alpha_2}x_{\beta}\ \ &&(x_\alpha=x_{\alpha_1}x_{\alpha_2})\\&=x_{\alpha_1}x_{\alpha_2'}x_\beta \ \ &&(x_{\alpha_2}x_{\beta}=x_{\alpha_2'}x_{\beta}) \\&=x_{\alpha'}x_{\beta}\ \ &&(x_{\alpha'}=x_{\alpha_1}x_{\alpha_2'}).   
            \end{alignedat}
        \end{align*}
\end{proof}

After knowing some words that can be obtained from $x_\alpha x_\beta$, we are allowed to replace certain letters in a given word using consequences of $\R_{r+1,r+2}$.

\begin{lemma}\label{lem: Tn: all length-2 words obtained from r+1,r+2}
    Let $\alpha,\beta$ be as defined above, and suppose that $\alpha'$,$\beta'$ are obtained from $\alpha,\beta$ using one of the following rules: \begin{thmenumerate}
        \item \label{lem: Tn: change word: change im beta} If $b_{r+1}'\in[n]\setminus\im\beta$, let $\beta'$ be obtained from $\beta$ by changing the image of $B_{r+1}$ from $b_{r+1}$ to $b_{r+1}'$, and let $\alpha'=\alpha$.
        \item \label{lem: Tn: change word: change im alpha} If $1\leq i\leq r+1$, $1\leq j\leq r$, $a_i\in B_j$, and $a_i'\in B_j\setminus\im\alpha$, let $\alpha'$ be obtained from $\alpha$ by changing the image of $A_i$ from $a_i$ to $a_i'$, and let $\beta'=\beta$.
        \item \label{lem: Tn: change word: change ker beta} If $1\leq i,j\leq r+1$, $|B_i\setminus\im\alpha|\geq 1$, and $b\in B_i\setminus\im\alpha$, let $\beta'$ be obtained from $\beta$ changing the kernel classes $B_i,B_j$ to $B_i\setminus\{b\}$ and $B_j\cup\{b\}$ respectively. Let $\alpha'=\alpha$.
        \item \label{lem: Tn: change word: change ker alpha} If $1\leq i\leq r-1$, $|A_i|\geq 2$, $a_{i1}\neq a_{i2}\in B_i$, let $\alpha'$ be obtained from $\alpha$ by breaking the kernel class $A_i$ into $A_{i1}$ and $A_{i2}$, mapping them to $a_{i1},a_{i2}$ respectively, and merging the two kernel classes whose images belong to the same $\ke\beta$ classes, namely $A_r,A_{r+1}$, into $A_r\cup A_{r+1}$, and mapping it to $a_r$ or $a_{r+1}$. Let $\beta'=\beta$.
    \end{thmenumerate}Then the word $x_{\alpha'} x_{\beta'}$ can be obtained from $x_\alpha x_\beta$ using $\R_{r+1,r+2}$.
\end{lemma}

We aim to show that ideals of $\T_n$ satisfy the conditions $\ref{P1}-\ref{P3}$. We start by showing that given a word of rank $r$ which is of length 2, we can obtain any word of length 2 that represents the same map, by applying the rules given in Lemma \ref{lem: Tn: all length-2 words obtained from r+1,r+2}.
\begin{proof}
    This follows from Lemmas \ref{lem: Tn: change im beta} - \ref{lem: Tn: change im alpha}.
\end{proof}

\begin{lemma}\label{lem: Tn: length-2 words}
    Let $\alpha,\beta,\gamma,\delta\in J_{r+1}$ such that $\alpha\beta=\gamma\delta\in J_r$. Then $x_\alpha x_\beta=x_\gamma x_\delta$ is a consequence of $\R_{r+1,r+2}$.
\end{lemma}

\begin{proof}
    Suppose that \[\alpha\beta=\gamma\delta=\begin{small}
        \begin{pmatrix}
        X_1 & X_2 & \dots & X_r\\
        x_1 & x_2 & \dots & x_r
    \end{pmatrix}
    \end{small},\] where $X_1,\dots,X_r$ partition $[n]$. Then $\{x_1,\dots,x_r\}\subseteq \im\beta$ and $\{x_1,\dots,x_r\}\subseteq \im\delta$. Let $\im\beta=\{x_1,\dots,x_r, y\}$, and $\im\delta=\{x_1,\dots,x_r, z\}$, where $y,z\in [n]\setminus\im\alpha\beta$. All $X_1,\dots,X_r$ are $\ke\alpha$-classes except for one of them, which is a union of two $\ke\alpha$-classes. A similar statement is true for $\ke\gamma$. Without loss of generality, assume that 
        \begin{align*}
        \alpha&=\begin{small}\begin{pmatrix}
            X_1 & X_2 & \dots & X_{i1} & X_{i2} & X_{i+1} & \dots & X_{r}\\
            a_1 & a_2 & \dots & a_{i1} & a_{i2} & a_{i+1} & \dots & a_{r}
        \end{pmatrix}\end{small},\\
        \beta&=\begin{small}\begin{pmatrix}
            B_1 & B_2 & \dots & B_i & \dots & B_{r-1} & B_r & B_{r+1}\\
            x_1 & x_2 & \dots & x_i & \dots & x_{r-1} & x_r & y
        \end{pmatrix}\end{small},
    \end{align*}where $X_{i1}, X_{i2}$ partition $X_i$, $B_1,\dots,B_{r+1}$ partition $[n]$, $a_j\in B_j$ for $j\in[r]\setminus\{i\}$; $a_{i1},a_{i2}\in B_i$, and $B_{r+1}\cap \im\alpha=\emptyset$. Similarly, 
        \begin{align*}
        \gamma&=\begin{small}\begin{pmatrix}
            X_1 & X_2 & \dots & X_{k1} & X_{k2} & X_{i+1} & \dots & X_{r}\\
            c_1 & c_2 & \dots & c_{k1} & c_{k2} & c_{i+1} & \dots & c_{r}
        \end{pmatrix}\end{small},\\
        \beta&=\begin{small}\begin{pmatrix}
            D_1 & D_2 & \dots & D_k & \dots & D_{r-1} & D_r & D_{r+1}\\
            x_1 & x_2 & \dots & x_k & \dots & x_{r-1} & x_r & z
        \end{pmatrix}\end{small},
    \end{align*}where $X_{k1}, X_{k2}$ partition $X_k$, and $c_{k1},c_{k2}\in D_k$. We will show how to use rules \ref{lem: Tn: change word: change im beta}-\ref{lem: Tn: change word: change ker alpha} in Lemma \ref{lem: Tn: all length-2 words obtained from r+1,r+2} to transform $x_\gamma x_\delta$ into $x_\alpha x_\beta$ as a consequence of $\R_{r+1,r+2}$. We split the proof into two cases, depending on whether $i=k$.

    \underline{Case 1}: $i=k$. We first show that we can obtain $x_\alpha x_{\delta'}$ from $x_\gamma x_{\delta}$, where $\im\delta'\im\delta$. We begin with $l=1$.
    
    \underline{Step 1}: If $a_l\in D_l$, go to Step 2. Otherwise $a_l\in D_j$ for some $j\neq l$. If $|D_j|>1$, change $\ke\delta$ by moving $a_l$ from $D_j$ to $D_l$ \ref{lem: Tn: change word: change ker beta}. Otherwise, $|D_j|=1$. In that case, change $\ke\delta$ by moving $c'\notin\{a_l\}\cup\im\gamma$ to $D_j$ \ref{lem: Tn: change word: change ker beta}. Then change $\im\gamma$ by swapping $c_j=a_l$ for $c'$ \ref{lem: Tn: change word: change im alpha}, and $\ke\delta$ by moving $a_l$ from $D_j$ to $D_l$ \ref{lem: Tn: change word: change ker beta}. 
    
    \underline{Step 2}: Change $\im\gamma$ by changing the image of $X_l$ from $c_l$ to $a_l$ \ref{lem: Tn: change word: change im alpha}. Define the resulting map to be $\gamma_l$ and $\delta'$.\\
    Next, we repeat the above until every element in $\im\gamma$ is changed to an element in $\im\alpha$ accordingly. To avoid ambiguity in the notation, we need to update the definition for $\gamma$ and $\delta$, as well as keeping track of the number of repetitions of Step 1 and 2. 
    
    \underline{Step 3}: If $l=r+1$, go to the next Step. Otherwise, let $\gamma=\gamma_l$, $\delta=\delta'$, $l=l+1$, and go back to Step 1.
    
    So far we have obtained $\gamma_{r+1}$, where $\im{\gamma_{r+1}}=\im\alpha$. Since $i=k$, we have obtained $x_\alpha x_{\delta'}$ from $x_\gamma x_\delta$. Finally, we show that we can obtain $x_\alpha x_\beta$ from $x_\alpha x_\delta'$. 
    
    \underline{Step 4}: Using rule \ref{lem: Tn: change word: change ker beta}, we may obtain $\beta'$ from $\delta'$ by moving elements between the kernel classes of $\delta'$, where $\ke{\beta'}=\ke\beta$.
    
    \underline{Step 5}: We obtain $\beta$ from $\beta'$, by changing the image of $B_{r+1}$ from $z$ to $y$ \ref{lem: Tn: change word: change im beta}.
    
    We have obtained $x_\alpha x_\beta$ from $x_\alpha x_{\delta'}$, completing the case.
    
    \underline{Case 2}: $i\neq k$. We show that we could obtain $\gamma'$ from $\gamma$ using the rules from the corollary, such that $\gamma'$ and $\alpha$ satisfy the condition for Case 1.
    
    \underline{Step 1}: If $|D_i|\neq1$, let $c\in D_i\setminus\{c_i\}$, and go to Step 2. Otherwise, if $|D_i|=1$, change $\ke\delta$ by moving $c\in [n]\setminus\im\gamma$ to $D_i$ \ref{lem: Tn: change word: change ker beta}.
    
    \underline{Step 2}: Rule \ref{lem: Tn: change word: change ker alpha} allows us to change the kernel classes of $\gamma$ and obtain $\gamma'$ such that \[\gamma'=\begin{small}
        \begin{pmatrix}
        X_1 & \dots & X_{i-1} & X_{i1} & X_{i2} & X_{i+1}\dots & X_k & \dots & X_r\\
        c_1 & \dots &c_{i-1} & c_i & c & c_{i+1}\dots & c_{k1} & \dots & c_r
    \end{pmatrix}
    \end{small}.\]
    
    \underline{Step 3}: Now $\gamma'$ and $\alpha$ satisfy the condition for Case 1. Apply the method in Case 1. 
\end{proof}

Recall that, if $n>m\geq 2$, $\dep{I_m}\geq 2$ by Proposition \ref{prop: lower bound for depth of all ideals}. It follows that a presentation for $I_m$ must contain at least relations of rank $m$ and $m-1$. As a result, from now on, we consider $m>\frac{n+1}{2}, 2\leq r\leq {\rm min}\{m-2,2m-n-1\}=2m-n-1$. 

\begin{lemma}\label{lem: Tn: length-2 and length-3}
    Let $1\leq r\leq 2m-n-1$. If $\alpha, \beta, \gamma\in J_{r+1},\alpha\beta, \beta\gamma, \alpha\beta\gamma\in J_{r}$, then a relation of the form $x_\alpha x_\beta x_\gamma=x_{\alpha'} x_\gamma$ where $\alpha'$ is some map in $J_{r+1}$, is a consequence of $\R_{r+1,r+2}$.
\end{lemma}

\begin{proof}
    Let \begin{align*}
        \alpha=\begin{small}\begin{pmatrix}
            A_1 & A_2 & \dots & A_{r+1}\\
            a_1 & a_2 & \dots & a_{r+1}
        \end{pmatrix}\end{small},\ \beta=\begin{small}\begin{pmatrix}
            B_1 & B_2 & \dots & B_{r+1}\\
            b_1 & b_2 & \dots & b_{r+1}
        \end{pmatrix}\end{small},\ \gamma=\begin{small}\begin{pmatrix}
            C_1 & C_2 & \dots & C_{r+1}\\
            c_1 & c_2 & \dots & c_{r+1}
        \end{pmatrix}\end{small}.
    \end{align*}Since $\alpha\beta,\beta\gamma,\alpha\beta\gamma\in D_{r}$, $\im\alpha$ intersects with $r$ of $\ke\beta$-classes, and $\im\beta,\im{\alpha\beta}$ intersect with $r$ of $\ke\gamma$-classes. Therefore the $\ke\gamma$-class that contains two elements in $\im\beta$ only contains one element from $\im{\alpha\beta}$. We may assume that $a_{r},a_{r+1}\in B_{r}$, and for $1\leq i\leq r-1$, $a_i\in B_i$, and $\im{\alpha\beta}=\{b_1,\dots,b_{r}\}$. Then if $i\neq j\leq r$, $b_i\gamma\neq b_j\gamma$. Hence we may let $b_i\in C_i$ for $i\leq r-1$, and $b_{r},b_{r+1}\in C_{r}$. Then \[\alpha\beta\gamma=\begin{small}
        \begin{pmatrix}
            A_1 & A_2 & \dots & A_{r-1} & A_r\cup A_{r+1}\\
            c_1 & c_2 & \dots & c_{r-1} & c_r
        \end{pmatrix}
    \end{small}.\]
    Let $b\in C_{r+1}$. Define \[\beta'=\begin{small}\begin{pmatrix}
            B_1 & B_2 & \dots & B_{r+1}\\
            b_1 & b_2 & \dots & b
        \end{pmatrix}\end{small},\alpha'=\begin{small}\begin{pmatrix}
            A_1 & A_2 & \dots & A_r & A_{r+1}\\
            b_1 & b_2 & \dots & b_r & b_{r+1}
        \end{pmatrix}
        \end{small}\in J_{r+1}.\] Then by Lemma \ref{lem: Tn: change im beta}, $\alpha\beta=\alpha\beta'$, and $x_\alpha x_\beta=x_\alpha x_{\beta'}$ is a consequence of $\R_{r+1,r+2}$. Notice that \[\beta'\gamma=\begin{small}\begin{pmatrix}
        B_1 & B_2 & \dots & B_{r+1}\\
            c_1 & c_2 & \dots & c_{r+1}
    \end{pmatrix}\end{small}\in J_{r+1}, \text{ and }\alpha'\gamma=\begin{small}
        \begin{pmatrix}
            A_1 & \dots & A_{r-1} & A_r\cup A_{r+1}\\
            c_1 & \dots & c_{r-1} & c_r
        \end{pmatrix}
    \end{small}=\alpha\beta'\gamma.\] Hence we have \begin{align*}
            \begin{alignedat}{2}
            x_\alpha x_\beta x_\gamma&=x_{\alpha}x_{\beta'}x_{\gamma}\ \ &&(x_\alpha x_\beta=x_\alpha x_{\beta'})\\&=x_{\alpha}x_{\beta'\gamma} \ \ &&(x_{\beta'}x_{\gamma}=x_{\beta'\gamma} {\rm\  is\ a\ relation\ in}\ \R_{r+1,r+2}) \\&=x_{\alpha'} x_{\gamma}\ \ &&({\rm Lemma} \ \ref{lem: Tn: all length-2 words obtained from r+1,r+2}).
            \end{alignedat}
        \end{align*}
\end{proof}

\begin{lemma}\label{lem: Tn: length-2 words deduced from length-2 words of higher rank}
    Let $1\leq r\leq 2m-n-1$. If $\gamma\delta\in J_{r}$ where $\rk\gamma,\rk\delta\geq r+1$, there exist $\alpha,\beta\in J_{r+1}$ such that $\alpha\beta=\gamma\delta$, and $x_\alpha x_\beta=x_\gamma x_\delta$ is a consequence of $\R_{r+1,m}$.
\end{lemma}

\begin{proof}
    Consider $\gamma,\delta$ where $\rk\gamma=i,\rk\delta=j$. Let \[\gamma=\begin{small}\begin{pmatrix}
        C_1 & C_2 & \dots & C_i\\
        c_1 & c_2 & \dots & c_i
    \end{pmatrix}\end{small}, \text{ and } \delta=\begin{small}\begin{pmatrix}
        D_1 & D_2 & \dots & D_j\\
        d_1 & d_2 & \dots & d_j
    \end{pmatrix}\end{small}.\] Since $\rk{\gamma\delta}=r$, the elements in $\im\gamma$ belong to $r$ of $\ke\delta$-classes, $D_1, D_2,\dots,D_{r}$, say. Relabel the $\ke\gamma$-classes and their images, \[\gamma=\begin{small}\begin{pmatrix}
        C_{11} & \dots & C_{1k_1} & C_{21} & \dots & C_{2k_2} & \dots & C_{r1} & \dots &C_{rk_r}\\
        c_{11} & \dots & c_{1k_1} & c_{21} & \dots & c_{2k_2} & \dots & c_{r1} & \dots &c_{r{k_r}}
    \end{pmatrix}\end{small},\] where $k_l=|D_l\cap\im\gamma|$, and $\{c_{l1},\dots,c_{lk_l}\}\subseteq{D_l}$ for $1\leq l \leq r$. Notice that with the assumption $r\leq 2m-n-1$, it is not possible that $\gamma,\delta\in J_m$ whereas $\gamma\delta\in J_r$. In the following, we can exclude the situation when $i=j=m$. 
    
\underline{Case 1}: $r+1 \leq i\leq j\leq m$.
        Let $c\in D_{r+1}$. Define \[\gamma_1=\begin{small}\begin{pmatrix}
        c_{11} & \dots & c_{1k_1} & \dots & c_{r1} & \dots &c_{rk_r} & [n]\setminus\{c_{11},\dots,c_{rk_r}\}\\
        c_{11} & \dots & c_{1k_1} &\dots & c_{r1} & \dots &c_{r{k_r}} & c
    \end{pmatrix}\end{small}\in J_{i+1},\] and \[\delta_1=\begin{small}\begin{pmatrix}
        \{c_{11},\dots,c_{1k_1}\} & \{c_{21},\dots,c_{2k_2}\} & \dots & \{c_{r1},\dots, c_{rk_r}\} & [n]\setminus\{c_{11},\dots,c_{rk_r}\}\\
        d_1 & d_2 & \dots & d_r & d_{r+1}
    \end{pmatrix}\end{small}.\] Then we have $\gamma\gamma_1=\gamma$, and $\gamma_1\delta=\delta_1$. Notice that $\gamma_1,\delta_1\in J_{r+1}\cup \dots\cup J_m$, hence \begin{align*}\begin{alignedat}{2}
        x_\gamma x_\delta&= x_\gamma x_{\gamma_1}x_\delta \ \ &&(x_\gamma=x_\gamma x_{\gamma_1}\ {\rm\ is\ a\ relation\ in\ } \R_{r+1,m})\\
        &=x_\gamma x_{\delta_1} \ \ &&(x_{\delta_1}=x_{\gamma_1} x_{\delta}\ {\rm\ is\ a\ relation\ in\ } \R_{r+1,m})
    \end{alignedat}
    \end{align*}is a consequence of $\R_{r+1,m}$.
    
\underline{Case 2}: $r+1\leq j < i \leq m$\\
Since $i>r+1$, there exists a $\ke\delta$-class that contains at least $2$ elements of $\im\gamma$. Here, we consider the situation where $k_r\geq 2$, and the proofs for other cases are similar. Let $c_{r1},c_{r2}\in D_r$. For $r+1\leq i\leq j$, let $d_i'\in D_i$. Define \begin{align*}
    \delta_1&=\begin{small}\begin{pmatrix}
    D_1 & \dots & D_{r-1} & c_{r1} & D_r\setminus \{c_{r1}\} & D_{r+1} & \dots & D_{j}\\
    c_{11} & \dots & c_{(r-1)1} & c_{r1} & c_{r2} & d_{r+1}' & \dots & d_j'
\end{pmatrix}\end{small},\\ \gamma_1&=\begin{small}\begin{pmatrix}
    \cup^{k_1}_{i=1}C_{1i}  & \dots & \cup^{k_{r-1}}_{i=1}C_{(r-1)i} &  C_{r1} & \cup^{k_r}_{i=2}C_{ri}\\c_{11} & \dots & c_{(r-1)1} & c_{r1} & c_{r2}
\end{pmatrix}\end{small}.
\end{align*} We have $\delta_1\delta=\delta$, $\gamma\delta_1=\gamma_1$, and $\gamma_1\in J_{r+1},\ \delta_1\in J_{j+1}$. Hence $x_\gamma x_\delta= x_\gamma x_{\delta_1}x_\delta=x_{\gamma_1}x_\delta$ is a consequence of $\R_{r+1,m}$.

Now if $\gamma\in D_i$ and $\delta\in D_j$ such that $r+1 \leq i\leq j\leq m$, apply the method in Case 1 to $x_\gamma x_\delta$ and obtain a word $u$. If $u$ contains a letter that represents a map of rank greater than $s$, apply the method in Case 2 to $u$. The resulting word should be made up of two letters that represent maps of rank $r+1$. Otherwise, if $\gamma\in D_i$ and $\delta\in D_j$ such that $r+1\leq j < i \leq m$, apply Case 1 to $x_\gamma x_\delta$, then Case 2 to the resulting word.
\end{proof}

\begin{prop}\label{prop: Tn: upper bound for Tn}
    Consider $I_m$ an ideal of $\T_n$. If $n>m>\frac{n+1}{2}$, then $\dep{I_m}\leq n-m+1$.
\end{prop}
\begin{proof}
    We have proved in Lemma \ref{lem: Tn: length-2 words}-\ref{lem: Tn: length-2 words deduced from length-2 words of higher rank} that when $m>\frac{n+1}{2}$, $\ref{P1}-\ref{P3}$ hold. It follows from Theorem \ref{thm: All: presentation for Im} that $\dep{I_m}\leq m-(2m-n)+1=n-m+1$.
\end{proof}

\begin{thm}\label{thm: Tn: relational depth}
    Let $I_m$ be an ideal of $\T_n$. Then $\dep{I_m}=\begin{cases}
        3 & m=n\\
        n-m+1 & m>\frac{n+1}{2}\\
        m & m\leq \frac{n+1}{2}
    \end{cases}$.
\end{thm}
\begin{proof}
    When $m\leq\frac{n+1}{2}$, the exact value for $\dep{I_m}$ is given in Corollary \ref{cor: All: exact value for depth of small ideals}. When $n>m>\frac{n+1}{2}$, the lower bound for $\dep{I_m}$ is found in Proposition \ref{prop: lower bound for depth of all ideals}, which is shown to be an upper bound in Proposition \ref{prop: Tn: upper bound for Tn}. 
    
    Finally, we know that $\dep{\T_n}\leq 3$, as $\dep{I_{n-1}}=2$. Consider transformations \begin{align*}
        \alpha=\begin{small}
        \begin{pmatrix}
            1 & 2 & \dots & n-2 & \{n-1,n\}\\
            1 & 2 & \dots & n-2 & n-1
        \end{pmatrix}
    \end{small}, \beta=\begin{small}
        \begin{pmatrix}
            1 & \dots & n-3 & \{n-2, n-1\} & n\\
            1 & \dots & n-3 & n-2 & n-1
        \end{pmatrix}
    \end{small}. \end{align*}Then \[\alpha\beta^2=\alpha\beta=\beta^2=\begin{small}
        \begin{pmatrix}
            1 & 2 & ... & n-3 & [n-2,n]\\
            1 & 2 & ... & n-3 & n-2
        \end{pmatrix}
    \end{small}\in J_{n-2}.\] However, we cannot obtain $x_\alpha{x_\beta}^2$ from ${x_\beta}^2$ using relations from $\R_{n-1,n}$, as $\alpha, \beta$ satisfy the condition in Proposition \ref{prop: general: an example of relations that cannot be deduced}. Hence $\C_{n-1}=\langle X_{n-1,n}\ |\ \R_{n-1,n}\rangle$ does not define a presentation for $\T_n$, and $\dep{\T_n}>2$, which gives $\dep{\T_n}=3$.
\end{proof}

\section{Upper bound for $\PT_n$}\label{sec: PT_n}
Finally, we look at ideals of $\PT_n$. 
Assume that, unless stated otherwise,
    \begin{align*}
    \alpha &= \begin{small}\begin{pmatrix}
        A_1 & A_2 & \dots & A_r & A_{r+1} \\
        a_1 & a_2 & \dots & a_r & a_{r+1} 
        \end{pmatrix}\end{small},\ 
        \beta = \begin{small}\begin{pmatrix}
        B_1 & B_2 & \dots & B_r & B_{r+1} \\
        b_1 & b_2 & \dots & b_r & b_{r+1} 
        \end{pmatrix}\end{small}.
\end{align*}Define $A=[n]\setminus\dm{\alpha}, B=[n]\setminus\dm{\beta}$.
We consider the two possible ways of obtaining a partial map of rank $r$ from a composition of two rank-$(r+1)$ partial maps.

\underline{Case 1}: $\im{\alpha}\subseteq\dm{\beta}$. In this case both $a_r,a_{r+1}\in B_r$, and
\[\alpha\beta=\begin{small}\begin{pmatrix}
        A_1 & A_2 & \dots &A_{r-1} & A_r\cup A_{r+1} \\
        b_1 & b_2 & \dots & b_{r-1} & b_r 
        \end{pmatrix}\end{small}.\]  

\underline{Case 2}: $|\im{\alpha}\ \cap\dm{\beta}^c|=1$. In this case we have $a_{r+1}\in B$, and
\[\alpha\beta=\begin{small}\begin{pmatrix}
        A_1 & A_2 & \dots &A_{r-1} & A_r\\
        b_1 & b_2 & \dots & b_{r-1} & b_r 
        \end{pmatrix}\end{small}.\]

$\T_n$ has a similar pattern to Case 1, and $\I_n$ is like Case 2. In fact, some of the results that were established in Sections \ref{sec: I_n} and \ref{sec: T_n} can also be applied to $\PT_n$ with some generalisations. Therefore we will not repeat the arguments that are the same in the previous sections. First we consider Case 1.
\begin{lemma}\label{lem: PTn: case 1 change im b}
    Let $b_{r+1}'\in [n]\setminus\im\beta$. Define \[\beta' = \begin{small}
        \begin{pmatrix}
        B_1 & B_2 & \dots & B_r & B_{r+1} \\
        b_1 & b_2 & \dots & b_r & b_{r+1}'
    \end{pmatrix}
    \end{small}.\] Then $\alpha\beta=\alpha\beta'$, and the relation $x_\alpha x_\beta=x_\alpha x_{\beta'}$ is a consequence of $\R_{r+1,r+2}$.
\end{lemma}

\begin{proof}
    The proof for Lemma \ref{lem: Tn: change im beta} also works here, with the only difference being the definition of $\beta_1'$, which is now a partial map. Here, under the same definition of $b_r'$ and $b_{r+1}'$, the rest of the proof is identical to that of Lemma \ref{lem: Tn: change im beta}.
\end{proof}

\begin{lemma}\label{lem: PTn: case 1 change ker b}
    If $|B_{r+1}|\geq 2$, let $b\in B_{r+1}$. Let $1\leq i\leq r+1$. Define \[\beta' = \begin{small}
        \begin{pmatrix}
        B_1 & B_2 &\dots &B_i\cup \{b\}&\dots & B_{r} & B_{r+1}\setminus\{b\} \\
        b_1 & b_2 &\dots  &b_i &\dots& b_{r} & b_{r+1} \end{pmatrix}
    \end{small}.\] Then $\alpha\beta=\alpha\beta'$, and the relation $x_\alpha x_\beta=x_\alpha x_{\beta'}$ is a consequence of $\R_{r+1,r+2}$. 
\end{lemma}

\begin{proof}
    The proof is the same as that of Lemma \ref{lem: change ker beta}. When defining $\beta_1$ and $\beta_1'$, we let $B$ be the complement of their domains.
\end{proof}

%As a consequence, the following lemma holds.
\begin{lemma}\label{lem: PTn: case 1 change ker b coro}
    For $1\leq i\neq j\leq r+1$, if $|B_i\setminus \im\alpha|\geq 1$, let $b\in B_i\setminus \im\alpha$, define \[\beta_{ij}' = \begin{small}
        \begin{pmatrix}
        B_1 & B_2 &\dots &B_i\setminus\{b\}& \dots& B_j\cup\{b\}&\dots & B_{r} & B_{r+1} &B\\
        b_1 & b_2 &\dots  &b_i &\dots& b_j & \dots& b_{r} & b_{r+1} &-\end{pmatrix}
    \end{small}.\] Then $\alpha\beta=\alpha\beta_{ij}'$, and the relation $x_\alpha x_\beta=x_\alpha x_{\beta_{ij}'}$ is a consequence of $\R_{r+1,r+2}$. 
\end{lemma}

\begin{proof}
    This follows from Lemma \ref{lem: PTn: case 1 change ker b}.
\end{proof}

\begin{lemma}\label{lem: PTn: case 1 change ker a}
    If $|A_r|\geq 2$, let $a\in A_r$. Define \[\alpha' = \begin{small}
        \begin{pmatrix}
        A_1 & A_2 & \dots & A_r\setminus\{a\} & A_{r+1}\cup \{a\}\\
        a_1 & a_2 & \dots & a_r & a_{r+1}
    \end{pmatrix}
    \end{small}.\] Then $\alpha\beta=\alpha'\beta$, and the relation $x_\alpha x_\beta=x_{\alpha'} x_\beta$ is a consequence of $\R_{r+1,r+2}$.
\end{lemma}

\begin{proof}
    Similar to the proof strategy of the above lemmas, the proof for Lemma \ref{lem: Tn: change ker alpha 1} can be applied to maps in $\PT_n$, with the consideration of the complements $A$ and $B$ of the domains in the definitions of $\alpha_1,\alpha_2$ and $\beta_1$.
\end{proof}

%For the following 2 lemmas, we follow the same proof for Lemma \ref{lem: Tn: change ker alpha 2} and Lemma \ref{lem: Tn: change im alpha} respectively, where we define $\alpha_1$ to account the compliment of the domain $A$.

\begin{lemma}\label{lem: PTn: case 1 change im a1}
    Let $1\leq i\leq r-1$. If $|A_i|\geq 2$ and $|B_i|\geq 2$, let $a_{i1},a_{i2}\in B_i$, let $A_{i1},A_{i2}$ partition $A_i$. Define \[\alpha'=\begin{small}\begin{pmatrix}
        A_1 & A_2 & \dots & A_{i1} & A_{i2} &\dots  & A_r\cup A_{r+1}\\
        a_1 & a_2 & \dots &a_{i1} & a_{i2}&\dots & a_r 
        \end{pmatrix}\end{small}.\] Then $\alpha\beta=\alpha'\beta$, and the relation $x_\alpha x_\beta=x_{\alpha'} x_\beta$ is a consequence of $\R_{r+1,r+2}$. 
\end{lemma}
\begin{proof}
    The proof is the same as that of Lemma \ref{lem: Tn: change ker alpha 2}, except that $\alpha_1$ is defined to be a partial map and that $A=[n]\setminus\dm{\alpha}$.
\end{proof}

\begin{lemma}\label{lem: PTn: case 1 change im a2}
    Let $1\leq i\leq r+1$. If $|B_i\setminus(B_i\cap \im\alpha)|\geq 1$, let $a_i'\in B_i\setminus \im\alpha$. Define \[\alpha' = \begin{small}
        \begin{pmatrix}
        A_1 & A_2 & \dots & A_i & \dots & A_r & A_{r+1} \\
        a_1 & a_2 & \dots & a_i' & \dots & a_r & a_{r+1} 
    \end{pmatrix}
    \end{small}.\] Then $\alpha\beta=\alpha'\beta$, and the relation $x_\alpha x_\beta=x_{\alpha'} x_\beta$ is a consequence of $\R_{r+1,r+2}$.
\end{lemma}
\begin{proof}
    The proof is the same as that of Lemma \ref{lem: Tn: change im alpha}, except that $\alpha_1$ is defined to be a partial map and that $A=[n]\setminus\dm{\alpha}$.
\end{proof}

\begin{lemma}\label{lem: PTn: type 1 words of length 2 are equal}
    Let $\alpha,\beta,\gamma,\delta\in J_{r+1}$ such that $\alpha\beta=\gamma\delta\in J_r$ are of Case 1. Then $x_\alpha x_\beta=x_\gamma x_\delta$ is a consequence of $\R_{r+1,r+2}$.
\end{lemma}

\begin{proof}
    To prove the lemma, it suffices to follow the proof for Lemma \ref{lem: Tn: all length-2 words obtained from r+1,r+2} on applying Lemmas \ref{lem: PTn: case 1 change im b} to \ref{lem: PTn: case 1 change im a2}.
\end{proof}

Next, we look at words that are obtained as in Case 2.
\begin{lemma}\label{lem: PTn: case 2 change ker a}
    Suppose that $|A_{r+1}|\geq 2$. Let $a'\in A_{r+1}$. Define \[\alpha' = \begin{small}
        \begin{pmatrix}
        A_1 & A_2 & \dots & A_r & A_{r+1}\setminus\{a\}\\
        a_1 & a_2 & \dots & a_r & a_{r+1}
    \end{pmatrix}
    \end{small}.\] Then $\alpha\beta=\alpha'\beta$, and the relation $x_\alpha x_\beta=x_{\alpha'} x_\beta$ is a consequence of $\R_{r+1,r+2}$.
\end{lemma}
\begin{proof}
    Let $a\in [n]\setminus\im\alpha, a''\in B_{r+1}$. Define \begin{align*}
        \alpha_1 &= \begin{small}
            \begin{pmatrix}
                A_1 & \dots & A_r & A_{r+1}\setminus\{a\} & a \\
                a_1 & \dots & a_r & a_{r+1} & a' 
            \end{pmatrix}
        \end{small},\  \alpha_2 = \begin{small}
            \begin{pmatrix}
                a_1 & \dots & a_r & \{a_{r+1},a'\} & a'' \\
                a_1 & \dots & a_r & a_{r+1} & a'' 
            \end{pmatrix}
        \end{small},\\ \alpha_2' &= \begin{small}
            \begin{pmatrix}
                a_1 & \dots & a_r & a_{r+1} & a''\\
                a_1 & \dots & a_r & a_{r+1} & a'' 
            \end{pmatrix}
        \end{small}.
    \end{align*}Then $\alpha_1\alpha_2=\alpha$, $\alpha_1\alpha_2'=\alpha'$, and \[\alpha_2\beta=\alpha_2'\beta=\begin{small}
        \begin{pmatrix}
            a_1 & \dots & a_r & a'' \\
            b_1 & \dots & b_r & b_{r+1} 
        \end{pmatrix}
    \end{small}\in J_{r+1}.\] Hence \begin{align*}
            \begin{alignedat}{2}
            x_\alpha x_\beta&=x_{\alpha_1}x_{\alpha_2}x_{\beta}\ \ &&(x_\alpha=x_{\alpha_1}x_{\alpha_2})\\&=x_{\alpha_1}x_{\alpha_2'}x_{\beta}\ \ &&(x_{\alpha_2}x_{\beta}=x_{\alpha_2'}x_{\beta}) \\&=x_{\alpha'} x_{\beta}\ \ &&(x_{\alpha'}=x_{\alpha_1}x_{\alpha_2'}).
            \end{alignedat}
        \end{align*}
\end{proof}

\begin{lemma}\label{lem: PTn: case 2 change im a}
    Suppose that $|B|\geq 2$. Let $a_{r+1}'\in B\setminus\{a_{r+1}\}$. Define \[\alpha' = \begin{small}
        \begin{pmatrix}
        A_1 & A_2 & \dots & A_r & A_{r+1} \\
        a_1 & a_2 & \dots & a_r & a_{r+1}'
    \end{pmatrix}
    \end{small}.\] Let $1\leq i\leq r$. Suppose that $|B_i|\geq 2$, let $a_i'\in B_i\setminus\{a_i\}$. Define \[\alpha'' = \begin{small}
        \begin{pmatrix}
        A_1 & \dots & A_i & \dots & A_r & A_{r+1}\\
        a_1 & \dots & a_i' & \dots & a_r & a_{r+1}
    \end{pmatrix}
    \end{small}.\] 
    Then $\alpha\beta=\alpha'\beta$, $\alpha\beta=\alpha''\beta$, and the relation $x_\alpha x_\beta=x_{\alpha'} x_\beta,x_\alpha x_\beta=x_{\alpha''} x_\beta$ are consequences of $\R_{r+1,r+2}$.
\end{lemma}
\begin{proof}
    Let $b\in [n]\setminus\im\beta$. Define\begin{align*}
        \beta_1&=\begin{small}
            \begin{pmatrix}
                B_1 & \dots & B_{r+1} & \{a_{r+1},a_{r+1}'\} \\
                b_1 & \dots & b_{r+1} & b 
            \end{pmatrix}
        \end{small},\ \beta_2= \begin{small}
            \begin{pmatrix}
                b_1 & \dots & b_{r+1} \\
                b_1 & \dots & b_{r+1} 
            \end{pmatrix}
        \end{small}.
    \end{align*}Then $\beta_1\beta_2=\beta, \alpha\beta_1=\alpha'\beta\in J_{r+1}$. We have \begin{align*}
            \begin{alignedat}{2}
            x_\alpha x_\beta&=x_{\alpha}x_{\beta_1}x_{\beta_2}\ \ &&(x_\beta=x_{\beta_1}x_{\beta_2})\\&=x_{\alpha'}x_{\beta_1}x_{\beta_2}\ \ &&(x_{\alpha}x_{\beta_1}=x_{\alpha'}x_{\beta_1}) \\&=x_{\alpha'} x_{\beta}\ \ &&(x_{\beta_1}x_{\beta_2}=x_{\beta}).
            \end{alignedat}
        \end{align*}For the second part, it suffices to prove the lemma when $i=1$. Let $a_1'\in B_1\setminus\{a_1\}$. Define\begin{align*}
            \alpha_1&= \begin{small}
                \begin{pmatrix}
                    a_1 & a_2 & \dots & a_{r+1} & a\\
                    a_1 & a_2 & \dots & a_{r+1} & a 
                \end{pmatrix}
            \end{small},\ \alpha_1'= \begin{small}
                \begin{pmatrix}
                    a_1 & a_2 & \dots & a_{r+1} & a \\
                    a_1' & a_2 & \dots & a_{r+1} & a
                \end{pmatrix}
            \end{small}.
        \end{align*}Then $\alpha\alpha_1=\alpha,\alpha\alpha_1'=\alpha''$, and $\alpha_1\beta=\alpha_1'\beta\in J_{r+1}$. Hence \begin{align*}
        \begin{alignedat}{2}
            x_\alpha x_\beta&=x_{\alpha}x_{\alpha_1}x_{\beta}\ \ &&(x_\alpha=x_{\alpha}x_{\alpha_1})\\&=x_{\alpha}x_{\alpha_1'}x_{\beta}\ \ &&(x_{\alpha_1}x_{\beta}=x_{\alpha_1'}x_{\beta}) \\&=x_{\alpha''} x_{\beta}\ \ &&(x_{\alpha}x_{\alpha_1'}=x_{\alpha''}).
            \end{alignedat}
        \end{align*}
\end{proof}

\begin{lemma}\label{lem: PTn: case 2 change ker b}
    Suppose that $|B|\geq 2$. Let $b\in B\setminus\{a_{r+1}\}$. For $1\leq i\leq r+1$, define \[\beta' = \begin{small}
        \begin{pmatrix}
        B_1 & B_2 &\dots &B_i\cup \{b\}&\dots & B_{r} & B_{r+1}\\
        b_1 & b_2 &\dots  &b_i &\dots& b_{r} & b_{r+1} \end{pmatrix}
    \end{small}.\] Then $\alpha\beta=\alpha\beta'$, and the relation $x_\alpha x_\beta=x_\alpha x_{\beta'}$ is a consequence of $\R_{r+1,r+2}$. 
\end{lemma}
\begin{proof}
Without loss of generality, let $i=1$. Let $a\in B_{r+1}$. Define \begin{align*}
    \begin{alignedat}{2}
    \beta_1&=\begin{small}
        \begin{pmatrix}
            B_1 & \dots & B_r & B_{r+1} & b \\
            a_1 & \dots & a_r & a & a_{r+1} 
        \end{pmatrix}
    \end{small},\ \beta_1'=\begin{small}
        \begin{pmatrix}
            B_1 & \dots & B_r & B_{r+1} & b \\
            a_1 & \dots & a_r & a & b 
        \end{pmatrix}
    \end{small},\\ \beta_2&=\begin{small}
        \begin{pmatrix}
            \{a_1,b\} & a_2 & \dots & a_{r+1} \\
            b_1 & b_2 & \dots & b_{r+1} 
        \end{pmatrix}
    \end{small},\ \alpha_1=\begin{small}
        \begin{pmatrix}
            a_1 & \dots & a_{r+1} & a \\
            a_1 & \dots & a_{r+1} & a 
        \end{pmatrix}
    \end{small}.
    \end{alignedat}
\end{align*}Then $\beta_1\beta_2=\beta, \beta_1'\beta_2=\beta'$, $\alpha\alpha_1=\alpha$, and $\alpha_1\beta_1=\alpha_1\beta_1'\in J_{r+1}$. Then \begin{align*}
            \begin{alignedat}{2}
            x_\alpha x_\beta&=x_{\alpha}x_{\alpha_1}x_{\beta_1}x_{\beta_2}\ \ &&(x_\alpha=x_{\alpha}x_{\alpha_1}, x_\beta=x_{\beta_1}x_{\beta_2})\\&=x_{\alpha}x_{\alpha_1}x_{\beta_1'}x_{\beta_2}\ \ &&(x_{\alpha_1}x_{\beta_1}=x_{\alpha_1}x_{\beta_1'}) \\&=x_{\alpha} x_{\beta'}\ \ &&(x_{\beta_1'}x_{\beta_2}=x_{\beta'}).
            \end{alignedat}
        \end{align*}
\end{proof}

\begin{lemma}\label{lem: PTn: case 2 change im b}
    Let $b_{r+1}'\in [n]\setminus\im\beta$. Define \[\beta' = \begin{small}
        \begin{pmatrix}
        B_1 & B_2 & \dots & B_r & B_{r+1} &B\\
        b_1 & b_2 & \dots & b_r & b_{r+1}' &-
    \end{pmatrix}
    \end{small}.\] Then $\alpha\beta=\alpha\beta'$, and the relation $x_\alpha x_\beta=x_\alpha x_{\beta'}$ is a consequence of $\R_{r+1,r+2}$.
\end{lemma}
\begin{proof}
    Let $b\in [n]\setminus(\im{\beta}\cup\im{\beta'})$. Define\begin{align*}
        \alpha_1 &=\begin{small}
            \begin{pmatrix}
                a_1 & \dots & a_{r+1}\\
                a_1 & \dots & a_{r+1}
            \end{pmatrix}
        \end{small},\ \beta_1=\begin{small}
            \begin{pmatrix}
                B_1 & \dots & B_{r+1} & B\\
                b_1 & \dots & b_{r+1} & b
            \end{pmatrix}
        \end{small},\\ \beta_1'&=\begin{small}
            \begin{pmatrix}
                B_1 & \dots & B_{r+1} & B\\
                b_1 & \dots & b_{r+1}' & b
            \end{pmatrix}
        \end{small},\ \beta_2=\begin{small}
            \begin{pmatrix}
                b_1 & \dots & b_{r+1} & b_{r+1}' \\
                b_1 & \dots & b_{r+1} & b_{r+1}' 
            \end{pmatrix}
        \end{small}.
    \end{align*}Then \begin{align*}
            \begin{alignedat}{2}
            x_\alpha x_\beta&=x_{\alpha}x_{\alpha_1}x_{\beta_1}x_{\beta_2}\ \ &&(x_\alpha=x_{\alpha}x_{\alpha_1})\\&=x_{\alpha}x_{\alpha_1}x_{\beta_1'}x_{\beta_2}\ \ &&(x_{\alpha_1}x_{\beta_1}=x_{\alpha_1}x_{\beta_1'}) \\&=x_{\alpha} x_{\beta'}\ \ &&(x_{\beta_1'}x_{\beta_2}=x_{\beta'}).
            \end{alignedat}
        \end{align*}
\end{proof}

\begin{lemma}\label{lem: PTn: all case 2 length-2 words obtained from r+1,r+2}
    Let $\alpha,\beta$ be as defined above, and suppose that $\alpha'$,$\beta'$ are obtained from $\alpha,\beta$ using one of the following rules: \begin{thmenumerate}
        \item \label{lem: PTn: change word: change ker alpha} If $|A_{r+1}|\geq 2$, and $a\in A_{r+1}$, let $\alpha'$ be obtained from $\alpha$ by changing the kernel class $A_{r+1}$ by moving $a$ to $A$, and let $\beta'=\beta$.
        \item \label{lem: PTn: change word: change im alpha} If $|B|\geq 2$, and $a_{r+1}'\in B\setminus\{a_{r+1}\}$, let $\alpha'$ be obtained from $\alpha$ by changing the image of $A_{r+1}$ from $a_{r+1}$ to $a_{r+1}'$, and let $\beta'=\beta$.
        \item \label{lem: PTn: change word: change im alpha 2} For $1\leq i\leq r$, if $|B_i|\geq 2$, and $a_i'\in B\setminus\{a_i\}$, let $\alpha'$ be obtained from $\alpha$ by changing the image of $A_i$ from $a_i$ to $a_i'$, and let $\beta'=\beta$.
        \item \label{lem: PTn: change word: change ker beta} For $1\leq i\leq r$, if $|B|\geq 2$, and $b\in B\setminus\{a_{r+1}\}$, let $\beta'$ be obtained from $\beta$ by changing the kernel by moving $b$ from $B$ to $B_i$, and let $\alpha'=\alpha$.
        \item \label{lem: PTn: change word: change im beta} If $b_{r+1}'\in [n]\setminus\im\beta$, let $\beta'$ be obtained from $\beta$ by changing the image of $B_{r+1}$ from $b_{r+1}$ to $b_{r+1}'$, and let $\alpha'=\alpha$.
    \end{thmenumerate}Then the word $x_{\alpha'} x_{\beta'}$ can be obtained from $x_\alpha x_\beta$ using $\R_{r+1,r+2}$.
\end{lemma}
\begin{proof}
    This follows from Lemmas \ref{lem: PTn: case 2 change ker a}-\ref{lem: PTn: case 2 change im b}.
\end{proof}

Using Lemma \ref{lem: PTn: all case 2 length-2 words obtained from r+1,r+2}, we aim to show that given a word $u$ of length 2 as in Case 2, we can obtain any other word $v$ which represents the same partial map, if $v$ is also of Case 2.
\begin{lemma}\label{lem: PTn: type 2 words of length 2 are equal}
    Let $\alpha,\beta,\gamma,\delta\in J_{r+1}$ such that $\alpha\beta=\gamma\delta\in J_r$ are of type $2$. Then $x_\alpha x_\beta=x_\gamma x_\delta$ is a consequence of $\R_{r+1,r+2}$.
\end{lemma}
\begin{proof}
    Suppose that \[\alpha\beta=\gamma\delta=\begin{small}
        \begin{pmatrix}
            A_1 & \dots & A_r & X\\
            b_1 & \dots & b_r & -
        \end{pmatrix}
    \end{small}.\] Without loss of generality, \begin{align*}
        \alpha&=\begin{small}
            \begin{pmatrix}
                A_1 & \dots & A_r & X_1 & X_1'\\
                a_1 & \dots & a_r & a_{r+1} & -
            \end{pmatrix}
        \end{small}, \ \beta=\begin{small}
            \begin{pmatrix}
                B_1 & \dots & B_r & B_{r+1} & B\\
                b_1 & \dots & b_r & b_{r+1} & -
            \end{pmatrix}
        \end{small},\\
        \gamma&= \begin{small}
            \begin{pmatrix}
                A_1 & \dots & A_r & X_2 & X_2'\\
                c_1 & \dots & c_r & c_{r+1} & -
            \end{pmatrix}
        \end{small},\ \delta=\begin{small}
            \begin{pmatrix}
                D_1 & \dots & D_r & D_{r+1} & D\\
                b_1 & \dots & b_r & d_{r+1} & -
            \end{pmatrix}
        \end{small},
    \end{align*}where $a_i\in B_i$ and $c_i\in D_i$ for $1\leq i\leq r$, $a_{r+1}\in B$ and $c_{r+1}\in D$. We will show that we can obtain $x_\alpha x_\beta$ from $x_\gamma x_\delta$ using rules in Lemma \ref{lem: PTn: all case 2 length-2 words obtained from r+1,r+2}. We first show that we can obtain a word of the form $x_\alpha x_{\delta'}$ from $x_\gamma x_\delta$. Let $i=1$.
    
    \underline{Step 1}: If $a_i\in D_i$, go to Step 2. Otherwise, $a_i\in D_j$. If $|D_j|>1$, change $\ke\delta$ by moving $a_i$ from $D_j$ to $D_i$ \ref{lem: PTn: change word: change ker beta}. If $|D_j|=1$, first move some element $c'\in [n]\setminus(\im\gamma\cup\{a_i\})$ from another $\ke\delta$-class to $D_j$ \ref{lem: PTn: change word: change ker beta}, then change $\gamma$ by changing the image of $C_j$ from $a_i$ to $c'$ \ref{lem: PTn: change word: change im alpha 2}, and change the kernel of $\delta$ by moving $a_i$ to $D_i$ \ref{lem: PTn: change word: change ker beta}. 
    
    \underline{Step 2}: Change $\gamma$ by changing the image of $A_i$ from $c_i$ to $a_i$ \ref{lem: PTn: change word: change im alpha 2}.
    
    \underline{Step 3}: If $i=r$, proceed. Otherwise let $i=i+1$, and return to Step 1.
    
    \underline{Step 4}: If $a_{r+1}\in D$, change $\im\gamma$ by changing the image of $X_2$ from $c_{r+1}$ to $a_{r+1}$ \ref{lem: PTn: change word: change im alpha}. Otherwise, if $a_{r+1}\notin D$, move elements between $D$ and $\ke\delta$-classes until $a_{r+1}$ is moved to $D$ \ref{lem: PTn: change word: change ker beta}, then change $\im\gamma$ as in the previous situation \ref{lem: PTn: change word: change im alpha}.
    
    \underline{Step 5}: Change $\delta$ by moving elements between $X_2$ and $X_2'$ until we have $X_1=X_2$ and $X_1'=X_2'$ \ref{lem: PTn: change word: change ker alpha}.
    
    So far we have obtained $x_\alpha x_{\delta'}$ from $x_\gamma x_\delta$.
    
    \underline{Step 6}: Change $\im{\delta'}$ by replacing the image of $X_1$ from $d_{r+1}$ by $b_{r+1}$ \ref{lem: PTn: change word: change im beta}.
    
    \underline{Step 7}: Change $\delta'$ by moving elements between the kernel classes, until we obtain $\beta$ \ref{lem: PTn: change word: change ker beta}.
\end{proof}

Then, we show that relations that equates words of Case 1 and Case 2 can be deduced from $\R_{r+1,r+2}$.
\begin{lemma}\label{lem: PTn: can switch between type 1 and 2}
    Suppose that $\alpha,\beta\in J_{r+1}$ such that $\alpha\beta\in J_r$ and is type 2. Suppose that $|B_r|\geq 2$. Let $a_{r1},a_{r2}\in B_r$, let $A_{r1},A_{r2}$ partition $A_r$. Define \[\alpha' = \begin{small}
        \begin{pmatrix}
        A_1 & A_2 & \dots&A_{r-1} & A_{r1} & A_{r2}\\
        a_1 & a_2 & \dots& a_{r-1} & a_{r1} & a_{r2}
    \end{pmatrix}
    \end{small}.\] Then $\alpha\beta=\alpha'\beta$, and the relation $x_\alpha x_\beta=x_{\alpha'} x_\beta$ is a consequence of $\R_{r+1,r+2}$.
\end{lemma}
\begin{proof}
    Let $b\in B_{r+1}$. Define \begin{align*}
        \alpha_1 &= \begin{small}
            \begin{pmatrix}
                A_1 & \dots & A_{r1} & A_{r2} & A_{r+1} \\
                a_1 & \dots & a_{r1} & a_{r2} & a_{r+1} 
            \end{pmatrix}
        \end{small},\ \alpha_2=\begin{small}
            \begin{pmatrix}
                a_1 & \dots & a_{r-1} & \{a_{r1},a_{r2}\} & a_{r+1} & a_r\\
                a_1 & \dots & a_{r-1} & a_r & a_{r+1} & b 
            \end{pmatrix}
        \end{small},\\ \alpha_2'&=\begin{small}
            \begin{pmatrix}
                a_1 & \dots & a_{r-1} & a_{r1} & a_{r2} & a_r \\
                a_1 & \dots & a_{r-1} & a_{r1} & a_{r2} & b 
            \end{pmatrix}
        \end{small}.
    \end{align*}Then $\alpha_1\alpha_2=\alpha,\alpha_1\alpha_2'=\alpha'$, and \[\alpha_2\beta=\alpha_2'\beta=\begin{small}
        \begin{pmatrix}
            a_1 & \dots & a_{r-1} & \{a_{r1},a_{r2}\} & a_{r} \\
            b_1 & \dots & b_{r-1} & b_r & b_{r+1}
        \end{pmatrix}
    \end{small}\in J_{r+1}.\] Hence \begin{align*}
            \begin{alignedat}{2}
            x_\alpha x_\beta&=x_{\alpha_1}x_{\alpha_2}x_{\beta}\ \ &&(x_\alpha=x_{\alpha_1}x_{\alpha_2})\\&=x_{\alpha_1}x_{\alpha_2'}x_{\beta}\ \ &&(x_{\alpha_2}x_{\beta}=x_{\alpha_2'}x_{\beta}) \\&=x_{\alpha'} x_{\beta}\ \ &&(x_{\alpha_1}x_{\alpha_2'}=x_{\alpha'}).
            \end{alignedat}
        \end{align*}
\end{proof}

\begin{lemma}\label{lem: PTn: all length-2 words that are equal}
    Let $\alpha,\beta,\gamma,\delta\in J_{r+1}$ such that $\alpha\beta=\gamma\delta\in J_r$. Then $x_\alpha x_\beta=x_\gamma x_\delta$ is a consequence of $\R_{r+1,r+2}$.
\end{lemma}
\begin{proof}
    If $\alpha\beta$ and $\gamma\delta$ are both of Case 1 or Case 2, the result follows from Lemma $\ref{lem: PTn: type 1 words of length 2 are equal}$ or Lemma $\ref{lem: PTn: type 2 words of length 2 are equal}$. Otherwise, $\alpha\beta$ and $\gamma\delta$ do not have the same type. Without loss of generality, assume that $\alpha\beta$ is as in Case 2. Then by Lemma $\ref{lem: PTn: can switch between type 1 and 2}$, we can obtain some $x_{\alpha'}x_\beta$ from $x_\alpha x_\beta$ such that $\alpha'\beta$ is Case 1. Finally, apply the result of Lemma $\ref{lem: PTn: type 1 words of length 2 are equal}$ to $x_{\alpha'}x_\beta$.
\end{proof}

\begin{lemma}\label{lem: PTn: length-3 words can be reduced to length-2}
    Let $0\leq r\leq 2m-n-1$. If $\alpha,\beta,\gamma\in J_{r+1},\alpha\beta,\beta\gamma,\alpha\beta\gamma\in J_{r}$, then there is a relation of the form $x_\alpha x_\beta x_\gamma=x_{\alpha'} x_\gamma$ where $\alpha'$ is some map in $J_{r+1}$, is a consequence of $\R_{r+1,r+2}$.
\end{lemma}

\begin{proof}
    Let \[\alpha=\begin{small}
        \begin{pmatrix}
            A_1 & \dots & A_{r+1}\\
            a_1 & \dots & a_{r+1}
        \end{pmatrix}
    \end{small}, \beta=\begin{small}
        \begin{pmatrix}
            B_1 & \dots & B_{r+1} \\
            b_1 & \dots & b_{r+1} 
        \end{pmatrix}
    \end{small}, \gamma=\begin{small}
        \begin{pmatrix}
            C_1 & \dots & C_{r+1} \\
            c_1 & \dots & c_{r+1} 
        \end{pmatrix}
    \end{small}.\] Since $\alpha\beta\gamma,\alpha\beta,\beta\gamma \in J_r$, the elements in $\im{\alpha\beta}$ are in $r$ of $\ke\gamma$-classes, and $\im\alpha$ are in $r$ of $\ke{\beta\gamma}$-classes. Therefore, for $1\leq i\leq r+1$, if $b_i\in \im{\alpha\beta}$, then $b_i\in \dm{\gamma}$. Moreover, $b_i\gamma\neq b_j\gamma$ if $b_i,b_j\in \im{\alpha\beta}$ for $i\neq j$. Without loss of generality, we may assume that either \[\alpha\beta=\begin{small}
        \begin{pmatrix}
            A_1 & \dots & A_r \\
            b_1 & \dots & b_r 
        \end{pmatrix}
    \end{small}, \text{\ or\ } \alpha\beta=\begin{small}
        \begin{pmatrix}
            A_1 & \dots & A_{r-1} & A_r\cup A_{r+1} \\
            b_1 & \dots & b_{r-1} & b_r 
        \end{pmatrix}
    \end{small}.\] Similarly, we have that \[\beta\gamma=\begin{small}
        \begin{pmatrix}
            B_1 & \dots & B_r \\
            c_1 & \dots & c_r
        \end{pmatrix}
    \end{small}, \text{\ or\ } \beta\gamma=\begin{small}
        \begin{pmatrix}
            B_1 & \dots & B_{r-1} & B_r\cup B_{r+1} \\
            c_1 & \dots & c_{r-1} & c_r
        \end{pmatrix}
    \end{small}.\] Define \[\alpha'=\begin{small}
        \begin{pmatrix}
            A_1 & \dots & A_r & A_{r+1} \\
            b_1 & \dots & b_r & b_{r+1}
        \end{pmatrix}
    \end{small},\ \beta'=\begin{small}
        \begin{pmatrix}
            B_1 & \dots & B_r & B_{r+1} \\
            b_1 & \dots & b_r & b_{r+1}'
        \end{pmatrix}
    \end{small},\] where $b_{r+1}'\in C_{r+1}$, and we consider every combination of $\alpha\beta$ and $\beta\gamma$. In fact, different combinations lead to the same proof. In the following we consider one possible combination, where \[\alpha\beta=\begin{small}
        \begin{pmatrix}
            A_1 & \dots & A_r \\
            b_1 & \dots & b_r 
        \end{pmatrix}
    \end{small},\text{\ and\ } \beta\gamma=\begin{small}
        \begin{pmatrix}
            B_1 & \dots & B_r \\
            c_1 & \dots & c_r
        \end{pmatrix}
    \end{small},\] and the proofs for the other three are identical.
    
    We have that $\alpha\beta\gamma=\begin{small}
        \begin{pmatrix}
            A_1 & \dots & A_r \\
            c_1 & \dots & c_r 
        \end{pmatrix}
    \end{small}$. Then $\beta'\gamma\in J_{r+1}$, and $\alpha\beta=\alpha\beta',\alpha'\gamma=\alpha(\beta'\gamma)$. We can obtain $x_{\alpha'}x_\gamma$ from $x_\alpha x_\beta x_\gamma$ via the following, \begin{align*}
            \begin{alignedat}{2}
            x_\alpha x_\beta x_\gamma&=x_{\alpha}x_{\beta'}x_{\gamma}\ \ &&(x_\alpha x_\beta=x_{\alpha}x_{\beta'})\\&=x_{\alpha}x_{\beta'\gamma}\ \ &&(\beta',\gamma,\beta'\gamma\in J_{r+1}) \\&=x_{\alpha'} x_{\gamma}\ \ &&(x_{\alpha}x_{\beta'\gamma}=x_{\alpha'}x_\gamma).
            \end{alignedat}
        \end{align*}
\end{proof}

\begin{lemma}\label{lem: PTn: length-2 words of higher rank}
    Let $0\leq r\leq 2m-n-1$. If $\gamma\delta\in J_{r}$ where $\rk\gamma,\rk\delta\geq r+1$, there exist $\alpha,\beta\in J_{r+1}$ such that $\alpha\beta=\gamma\delta$, and $x_\alpha x_\beta=x_\gamma x_\delta$ is a consequence of $\R_{r+1,m}$.
\end{lemma}

\begin{proof}
    Consider \[\gamma=\begin{small}
        \begin{pmatrix}
            C_1 & \dots & C_i\\
            c_1 & \dots & c_i
        \end{pmatrix}
    \end{small}\in J_i,\ \delta=\begin{small}
        \begin{pmatrix}
            D_1 & \dots & D_j \\
            d_1 & \dots & d_j 
        \end{pmatrix}
    \end{small}\in J_j,\] and let $D$ denote $[n]\setminus\dm{\delta}$. Since $\gamma\delta\in J_r$, there are at least $r$, say $s$ where $s\geq r$, elements in $\im\gamma$ that are in $r$ $\ke\delta$-classes, say $D_1,\dots,D_r$, and the rest of $\im\gamma$ is in $D$. Let $c\in D_{r+1}$. Relabel the $\ke\gamma$-classes and their images as the following, \[\gamma=\begin{small}
        \begin{pmatrix}
            C_{11} & \dots & C_{1k_1} & \dots & C_{r1} & \dots & C_{rk_r} & C_{(r+1)1} & \dots & C_{(r+1)(i-s)} \cr
            c_{11} & \dots & c_{1k_1} & \dots & c_{r1} & \dots & c_{rk_r} & c_{(r+1)1} & \dots & c_{(r+1)(i-s)} 
        \end{pmatrix}
    \end{small},\] where $k_l=|D_l\cap\im\gamma|$, $\{c_{l1},\dots,c_{lk_l}\}\subseteq D_l$ for $1\leq l\leq r$, and $c_{(r+1)1},\dots,c_{(r+1)(i-s)}\in D$. Notice that with the assumption $r\leq 2m-n-1$, there does not exist $\gamma,\delta\in J_m$ such that $\gamma\delta\in J_r$, therefore in the following, we exclude the situation when $i=j=m$. 
    
    \underline{Case 1}: $r+1\leq i\leq j\leq m$. Let $d\in D$. Define \[\gamma_1=\begin{small}
        \begin{pmatrix}
            c_{11} & \dots & c_{1k_1} & \dots & c_{(r+1)1} & \dots & c_{(r+1)(i-s)} & [n]\setminus\im\gamma\\
            c_{11} & \dots & c_{1k_1} & \dots & c_{(r+1)1} & \dots & c_{(r+1)(i-s)} & c
        \end{pmatrix}
    \end{small}\in J_{i+1},\] notice that the preimage of $c$ is non-empty as $i< n$. Define \[\delta_1=\begin{small}
        \begin{pmatrix}
            \{c_{11},\dots,c_{1k_1}\} & \dots & \{c_{r1},\dots,c_{rk_r}\} & [n]\setminus\im\gamma\\
            d_1 & \dots & d_r & d_{r+1}
        \end{pmatrix}
    \end{small}\in J_{r+1}.\] Then $\gamma\gamma_1=\gamma, \gamma_1\delta=\delta_1$. Hence we may obtain $x_\gamma x_\delta=x_\gamma x_{\gamma_1}x_\delta=x_\gamma x_{\delta_1}$ using relations in $\R_{r+1,m}$.
    
    \underline{Case 2}: $r+1\leq j<i\leq m$. Since $i>r+1$, there is a $\ke\delta$-class $D_l$, or $D$, that contains more than $1$ elements of $\im\gamma$. If $k_l\geq 2$ for some $1\leq l\leq r$, let $c_{l1},c_{l2}\in D_l$. It suffices to consider $l=r$, as the proof for the other possibilities should be similar. For $r+1\leq k\leq j$, let $d_k'\in D_k$. Define \[\delta_1=\begin{small}
        \begin{pmatrix}
            D_1 & \dots & D_{r-1} & c_{r1} & D_r\setminus \{c_{r1}\} & D_{r+1} & \dots & D_{j}\\
    c_{11} & \dots & c_{(r-1)1} & c_{r1} & c_{r2} & d_{r+1}' & \dots & d_j'
        \end{pmatrix}
    \end{small}\in J_{j+1},\] and \[\gamma_1=\begin{small}\begin{pmatrix}
    \cup^{k_1}_{i=1}C_{1i}  & \dots & \cup^{k_{r-1}}_{i=1}C_{(r-1)i} &  C_{r1} & \cup^{k_r}_{i=2}C_{ri} \\c_{11} & \dots & c_{(r-1)1} & c_{r1} & c_{r2} 
\end{pmatrix}\end{small}\in J_{r+1}.\] Otherwise, $k_l=1$ for all $1\leq l\leq r$, i.e. each of $D_1,\dots, D_r$ contains exactly one element of $\im\gamma$, and the rest $(i-r)$ elements of $\im\gamma$ are in $D$. Then \[\gamma=\begin{small}
        \begin{pmatrix}
            C_{11} & \dots & C_{r1} & C_{(r+1)1} & \dots & C_{(r+1)(i-r)} \cr
            c_{11} & \dots & c_{r1} & c_{(r+1)1} & \dots & c_{(r+1)(i-r)} 
        \end{pmatrix}
    \end{small}.\] Let $d\in D$. Define \[\delta_1=\begin{small}
        \begin{pmatrix}
            D_1 & \dots & D_r & D_{r+1} & \dots & D_{j} & c_{(r+1)1} \\
            c_{11} & \dots & c_{r1} & d_{r+1}' & \dots & d_j' & d 
        \end{pmatrix}
    \end{small}\in J_{j+1},\] and \[\gamma_1=\begin{small}
        \begin{pmatrix}
            C_{11} & \dots & C_{r1} & C_{(r+1)1} \\
            c_{11} & \dots & c_{r1} & d 
        \end{pmatrix}
    \end{small}\in J_{r+1}.\] In both situations, $\delta_1\delta=\delta$, $\gamma\delta_1=\gamma_1$, and $\gamma_1\in J_{r+1},\ \delta_1\in J_{j+1}$. Hence $x_\gamma x_\delta= x_\gamma x_{\delta_1}x_\delta=x_{\gamma_1}x_\delta$ is a consequence of $\R_{r+1,m}$.
    
Now if $\gamma\in D_i$ and $\delta\in D_j$ such that $r+1 \leq i\leq j\leq m$, apply the method in Case 1 to $x_\gamma x_\delta$ and obtain a word $u$. If $u$ contains a letter that represents a map of rank greater than $s$, apply Case 2 to $u$. The resulting word is made up of two letters that represent maps of rank $r+1$. Otherwise, if $\gamma\in D_i$ and $\delta\in D_j$ such that $r+1\leq j < i \leq m$, apply Case 2 to $x_\gamma x_\delta$, then Case 1 to the resultant word.
\end{proof}

\begin{prop}\label{prop: PTn: upper bound for PTn}
    Consider $I_m$ an ideal of $\PT_n$. If $n>m>\frac{n}{2}$, then $\dep{I_m}\leq n-m+1$.
\end{prop}
\begin{proof}
    We have proved in Lemma \ref{lem: PTn: all length-2 words that are equal}-\ref{lem: PTn: length-2 words of higher rank} that when $m>\frac{n}{2}$, $\ref{P1}-\ref{P3}$ hold. It follows from Theorem \ref{thm: All: presentation for Im} that $\dep{I_m}\leq m-(2m-n)+1=n-m+1$.
\end{proof}

\begin{thm}\label{thm: PTn: relational depth}
    Let $I_m< \PT_n$. Then $\dep{I_m}=\begin{cases}
        3 & m=n\\
        n-m+1 & m>\frac{n}{2}\\
        m+1 & m\leq \frac{n}{2}
    \end{cases}$.
\end{thm}

\begin{proof}
    When $m\leq\frac{n}{2}$, the exact value for $\dep{I_m}$ is given in Corollary \ref{cor: All: exact value for depth of small ideals}. When $n>m>\frac{n}{2}$, the lower bound for $\dep{I_m}$ is found in Proposition \ref{prop: lower bound for depth of all ideals}, which is shown to be an upper bound in Proposition \ref{prop: PTn: upper bound for PTn}. 
    
    Finally, we know that $\dep{\PT_n}\leq 3$, as $\dep{I_{n-1}}=2$. Consider the partial bijections $\alpha, \beta$ defined in the proof for Theorem \ref{thm: In: relational depth}. Since $\alpha, \beta$ are also elements of $\PT_n$, and that they satisfy the conditions in Proposition \ref{prop: general: an example of relations that cannot be deduced}, as a result we cannot obtain $x_\alpha{x_\beta}^2$ from ${x_\beta}^2$ using relations from $\R_{n-1,n}$. Hence $\C_{n-1}=\langle X_{n-1,n}\ |\ \R_{n-1,n}\rangle$ does not define a presentation for $\PT_n$, and $\dep{\PT_n}>2$, which gives $\dep{\PT_n}=3$.
\end{proof}

\section{Conclusion} \label{sec:conc}

Most work in proving our Main Theorem was to  show that the lower bounds for the depth of an ideal of $\T_n, \I_n, \PT_n$ are also upper bounds. Now let us mention the relationship between the lower bounds and another interesting concept. 

Let $S$ be a finite semigroup whose $\J$-classes form a chain, namely $J_\epsilon < \dots <J_m$. Suppose that $S$ does not contain a group of units. We define the \textit{multiplication depth} of $S$ to be $m-r+1$, where $r$ is the smallest number such that the product of two elements in $J_m$ can land in $J_r$. Together with Proposition \ref{prop: general: an example of relations that cannot be deduced} and Lemma \ref{lem: all: words that can be deduced from x_betax_gamma}, the definition suggests that the multiplication depth is a lower bound for the relational depth. The main theorem can be rewritten to capture this.

\begin{cor}
    Let $I_m$ be a proper ideal of $S\in\{\T_n, \I_n, \PT_n\}$. Then the relational depth of $I_m$ is equal to the multiplication depth.
\end{cor}

This follows easily upon recalling that
\[
\rank(\alpha) + \rank(\beta) \leq \rank(\alpha\beta) + n \quad \text{for any }\alpha,\beta\in\PT_n,
\]
the transformation version of the well know Sylvester rank inequality from linear algebra.

It is not true in general that the relational depth and the multiplication depth coincide. 

\begin{example}
    Let $I_m$ be a proper ideal of $\T_n$, where $m>\frac{n+1}{2}$, and $\dep{I_m}=m-r+1$, where $r=2m-n$. Define the Rees quotient semigroup $I_m'=I_m\setminus I_{r-1}$. Notice that the $\J$-classes of $I_m'$ are $\{0\}<J_{r}<J_{r+1}<\dots<J_m$. The multiplication depth of $I_m'$ is equal to that of $I_m$, and the reader may check that it is impossible for two transformations of rank $m$ to compose to one of rank less than or equal to $2m-n$. However, $\langle X_{r,m}\ |\ \R_{r,m}\rangle$ does not define $I_m'$, since $\langle X_{r,m}\ |\ \R_{r,m}\rangle=I_m$ and $I_m'$ is a different semigroup. Hence $\dep{I_m'}\neq\dep{I_m}$. In fact, the relational depth of $I_m'$ is the length of the chain of $\J$-classes, namely $\dep{I_m'}=m-r+2=n-m+2$.
\end{example}

There are some other families of semigroups whose $\J$-classes form a chain, for example, the partition monoid, and its subsemigroups, including Brauer monoid and Temperley--Lieb monoid, and the semigroups of linear transformations. We ask about the relational depth of the above semigroups, and whether they coincide with the multiplication depth.
It would also be of interest to investigate analogous problems for semigroups whose $\J$-classes form more complicated posets, not just chains. However, more thought is required in that case even to specify what one means by depth.

\bibliographystyle{plain}

\begin{thebibliography}{99}


\bibitem{Aizenstat58}
A.\ A\u{\i}zen\v{s}tat, Defining relations of finite symmetric semigroups, Mat. Sb. (N.S.) 45 (1958), 261--280.


\bibitem{CDEGZ}
S.\ Carson, I.\ Dolinka, J.\ East, V. Gould, R.\ Zenab, Product decompositions of semigroups induced by action pairs, Dissertationes Math. 587 (2023), 1--180. 

\bibitem{East06}
J.\ East, A presentation of the singular part of the symmetric inverse monoid, Comm. Algebra 34 (2006), 1671--1689.

\bibitem{East15}
J.\ East. A symmetrical presentation for the singular part of the symmetric inverse monoid, Algebra Universalis, 74 (2015), 207--228.

\bibitem{East10a}
J.\ East, A presentation for the singular part of the full transformation semigroup, Semigroup Forum 81 (2010), 357--379.

\bibitem{East13}
J.\ East. Defining relations for idempotent generators in finite full transformation semigroups, Semigroup Forum 86 (2013), 451--485.

\bibitem{East10b}
J.\ East. Presentations for singular subsemigroups of the partial transformation semigroup, Internat. J. Algebra Comput. 20 (2010), 1--25.

\bibitem{East14}
J.\ East. Defining relations for idempotent generators in finite partial transformation semigroups, Semigroup
Forum 89 (2014), 72--76.

\bibitem{CFTS}
O.\ Ganyushkin, V.\ Mazorchuk, Classical Finite Transformation Semigroups, Springer, London, 2008.

\bibitem{FST}
J.M.\ Howie, Fundamentals of Semigroup Theory, Clarendon Press, Oxford, 1995.

\bibitem{HM:IR}
J.M.\ Howie, R.B. McFadden, Idempotent rank in finite full transformation semigroups, Proc. Roy. Soc. Edinburgh 114 (1990), 161--167.


\bibitem{SCA}
G.\ Lallement, Semigroups and Combinatorial Applications, John Wiley $\&$ Sons, Inc., 1979.

\bibitem{Meakin93}
S.W.\ Margolis, J.C.\ Meakin, Free Inverse Monoid and Graph Immersions. International Journal of Algebra and Computation 3 (1993), 79--99.



\bibitem{MW:SP}
J.D.\ Mitchell, M.T.\ Whyte, Short presentation for transformation monoids, preprint, arXiv:2406.19294.

\bibitem{Popova61}
L.\ M.\ Popova, Defining relations in some semigroups of partial transformations of a finite set, Uchenye Zap. Leningrad. Gos. Ped. Inst.  218 (1961), 191--212.



\bibitem{NRthesis}
N.\ Ru\v{s}kuc, Semigroup Presentations, PhD thesis, University of St Andrews, 1995.

\bibitem{Sutov60}
\`E.\ G.\ \v{S}utov, Defining relations of finite semigroups of partial transformations, Dokl. Akad. Nauk SSSR 132 (1960), 1280--1282.



\bibitem{TGT}
G.C.\ Smith, O.M.\ Tabachnikova, Topics in Group Theory, Springer, London, 2000

\end{thebibliography}

\end{document}